\documentclass[preprint,12pt,authoryear]{elsarticle}
\journal{arXiv.org}

\usepackage{epsfig}
\usepackage{amssymb,amsmath,amsthm,url}
\usepackage{multirow}
\usepackage{booktabs}
\usepackage{natbib}
\usepackage{hyperref}
\usepackage{setspace}
\usepackage{dsfont}

\usepackage{amsfonts,times,eucal}
\usepackage{graphicx,ifthen}
\usepackage{wrapfig}
\usepackage{latexsym}
\usepackage{color}
\usepackage{xcolor}
\usepackage{pgf,tikz}

\usepackage{mathrsfs}
\usepackage{bm}
\usepackage{bbm}
\usepackage{srctex}
\usepackage{enumitem}
\usepackage{stackengine}
\stackMath
\makeatletter
\newcommand{\mylabel}[2]{#2\def\@currentlabel{#2}\label{#1}}
\makeatother
\usepackage{pgf,tikz,pgfplots}
\pgfplotsset{compat=1.15}
\usepackage{mathrsfs}
\usetikzlibrary{arrows}
\definecolor{zzttqq}{rgb}{0.6,0.2,0}
\definecolor{qqqqcc}{rgb}{0,0,0.8}
\definecolor{ffqqqq}{rgb}{1,0,0}
\definecolor{cqcqcq}{rgb}{0.7529411764705882,0.7529411764705882,0.7529411764705882}

\usepackage{mathtools}
\DeclarePairedDelimiter{\floor}{\lfloor}{\rfloor}

\theoremstyle{plain}
\newtheorem{theorem}{Theorem}[section]
\newtheorem{corollary}[theorem]{Corollary}
\newtheorem{lemma}[theorem]{Lemma}
\newtheorem{proposition}[theorem]{Proposition}

\newtheorem{condition}[theorem]{Condition}

\theoremstyle{definition}
\newtheorem{definition}[theorem]{Definition}

\newtheorem{example}[theorem]{Example}

\theoremstyle{remark}

\numberwithin{equation}{section}

\newcommand{\mA}{\mathbb{A}}
\newcommand{\mF}{\mathbb{F}}

\newcommand{\mM}{\mathbb{M}}
\newcommand{\N}{\mathbb{N}}
\newcommand{\R}{\mathbb{R}}

\newcommand{\mX}{\mathbb{X}}
\newcommand{\mY}{\mathbb{Y}}
\newcommand{\mZ}{\mathbb{Z}}
\newcommand{\Z}{\mathbb{Z}}

\newcommand{\p}{\mathbb{P}}

\newcommand{\cA}{\mathcal{A}}
\newcommand{\cB}{\mathcal{B}}
\newcommand{\cC}{\mathcal{C}}

\newcommand{\cE}{\mathcal{E}}

\newcommand{\cF}{\mathcal{F}}

\newcommand{\cI}{\mathcal{I}}
\newcommand{\cK}{\mathcal{K}}
\newcommand{\cL}{\mathcal{L}}

\newcommand{\cP}{\mathcal{P}}
\newcommand{\cR}{\mathcal{R}}

\newcommand{\fD}{\mathfrak{D}}

\newcommand{\fS}{\mathfrak{S}}
\newcommand{\fT}{\mathfrak{T}}
\newcommand{\fU}{\mathfrak{U}}

\newcommand{\sN}{\mathscr{N}}
\newcommand{\E}[1]{\mathbb{E}\left [  #1 \right ]}

\renewcommand{\epsilon}{\varepsilon}
\renewcommand{\phi}{\varphi}

\newcommand{\pspace}{(\Omega,\cA,\p)}

\newcommand{\intd}[1]{\mathrm{d}#1}

\newcommand{\norm}[1]{\left\lVert #1 \right\rVert}

\newcommand{\1}[1]{\,\mathds{1}\! \left\{ #1 \right\} }

\DeclareMathOperator*{\esssup}{ess\,sup}
\DeclareMathOperator*{\essinf}{ess\,inf}

\newcommand{\diam}{\operatorname{diam}}
\newcommand{\wt}{\widetilde}
\newcommand{\ol}{\overline}

\allowdisplaybreaks

\begin{document}

\begin{frontmatter}

\title{
On limit theorems for persistent Betti numbers from dependent data
}

\author[1]{Johannes Krebs}
\ead{krebs@uni-heidelberg.de}
\affiliation[1]{
			organization={Institute for Applied Mathematics, Heidelberg University},
            addressline={Im Neuenheimer Feld 205}, 
            city={Heidelberg},
            postcode={69120}, 
            country={Germany}
            }

\begin{abstract}
We study persistent Betti numbers and persistence diagrams obtained from time series and random fields. It is well known that the persistent Betti function is an efficient descriptor of the topology of a point cloud. So far, convergence results for the $(r,s)$-persistent Betti number of the $q$th homology group, $\beta^{r,s}_q$, were mainly considered for finite-dimensional point cloud data obtained from i.i.d.\ observations or stationary point processes such as a Poisson process. In this article, we extend these considerations. We derive limit theorems for the pointwise convergence of persistent Betti numbers $\beta^{r,s}_q$ in the critical regime under quite general dependence settings.
\end{abstract}

\end{frontmatter}

\vspace{3em}
Topological data analysis (TDA) is a comparably young field in (applied) mathematics at the intersection between computational geometry, probability theory, mathematical statistics and machine learning. 
Seminal papers which popularized the ideas of TDA are  \cite{edelsbrunner2000topological}, \cite{zomorodian2005computing}, \cite{carlsson2009topology}. An introduction offers the monograph of \cite{boissonnat2018geometric}. Statistical aspects of TDA are discussed in the surveys of \cite{chazal2017introduction} and \cite{bobrowski2018topology}.

 In this article, we will focus on a special topic in persistent homology, which itself is the major branch in TDA: We study the large sample behavior of persistent Betti numbers and the corresponding persistence diagram obtained from time series or random fields.

So far, the literature has focused on point cloud data obtained from two major sources. On the one hand, there are various limit theorems for persistent Betti numbers obtained from stationary point processes as a rather general class, a prominent example here is the homogeneous Poisson process. On the other hand, the binomial process (a sample of i.i.d.\ data) is intensely studied, too.

In early contributions, \cite{kahle2011random} investigates the asymptotic behavior of Betti numbers in the sub-, supercritical and critical regime. Extensions are given by \cite{kahle2013limit} and \cite{yogeshwaran2015topology}. From the above mentioned three asymptotic regimes, the critical (or thermodynamic) regime certainly gets the most attention and in the following, we will also limit the discussion in the introduction to this case.

One of the first major contributions which studies large deviation inequalities and central limit theorems for the Poisson and binomial sampling scheme in the critical regime is the work of \cite{yogeshwaran2017random}. Extensions to persistent Betti numbers and persistence diagrams are given in \cite{hiraoka2018limit}. \cite{trinh2018central} provides an abstract result for the asymptotic normality of Betti numbers. \cite{krebs2018asymptotic} study the stabilizing properties of and related central limit theorems from Betti numbers built from non homogeneous Poisson or binomial processes. Strong laws of large numbers for Betti numbers obtained from the Poisson or the binomial process on general manifolds are considered in \cite{goel2018strong}. Other recent contributions which also discuss limiting results for Betti numbers are \cite{owada2018limit}, \cite{owada2019limit}. \cite{divol2018persistence} study the limiting behavior of the persistence diagram.

In the context of time series, the behavior of Betti numbers has been mainly investigated in applications. \cite{islambekov2020harnessing} combine the TDA methodology with classical methods for change point detection. Classification problems for time series using methods from TDA are considered in \cite{seversky2016time} and in \cite{umeda2017time}. The applications of TDA to networks obtained from financial data are studied in \cite{gidea2017topology} and \cite{gidea2018topological}; here the methods of TDA measure a type of high-dimensional and time-dependent correlation in the network.

The persistence landscape (\cite{bubenik2015statistical}) is an efficient summary statistic of the persistence diagram and is quite popular in machine learning; we also refer to \cite{chazal2014stochastic} and \cite{kim2020efficient} for related contributions.

The aim of this paper, is to provide two advances in the study of persistent Betti numbers in the context of time series and random fields. On the one hand, we study the large sample behavior of the expectation of persistent Betti numbers obtained from time series and random fields. More precisely, for the time series case, let $X=(X_t: t\in\Z)\subseteq [0,1]^p$ be a stationary Markov chain of order $m$ (w.r.t.\ its natural filtration) with a continuous and strictly positive joint density $g$ of $(X_1,\ldots,X_{m+1})$. Write $\kappa$ for the marginal density of each $X_t$. It is well-known that for an $n$-binomial process $\mX^*_n$, which consists of $n$ i.i.d.\ observations $X^*_t$ with marginal density $\kappa$, the limit of $n^{-1} \E{ \beta^{r,s}_q(\cK(n^{1/p} \mX^*_n)) }$ exists. Using the nearly additive properties of persistent Betti numbers, we show that Markov chains converge to the same limit. In fact, denoting $\cK$ the \v Cech or Vietoris-Rips filtration, we have
\[
			\lim_{n\rightarrow \infty} n^{-1} \E{ \beta^{r,s}_q( \cK( n^{1/p} \mX_n) ) } = \E{ \hat{b}_q(\kappa(X_t)^{1/p}(r,s)) }, \quad \forall\, 0\le q\le p-1, \quad \forall\, 0\le r\le s<\infty,
\]
and where $\hat{b}_q(r,s)$ is the limit of $n^{-1} \E{ \beta^{r,s}_q(\cK(n^{1/p} \mY^*_n)) }$ for an $n$-binomial process $\mY^*_n$ on $[0,1]^p$ with uniform density $\kappa$.  We also prove a related strong law of large numbers. Doing so, we can also conclude convergence results for persistence diagrams. Moreover, we establish similar convergence results for stationary random fields.

On the other hand, we establish an exponential inequality and give strong laws of large numbers for persistent Betti numbers, which are not exclusively derived from point clouds on $\R^p$. Instead, we also allow for functional data as a potential data source. The presented exponential inequality relies on the concept of the Marton coupling, see \cite{marton2003measure}. Marton couplings have also been successfully used in the past to derive concentration inequalities of the McDiarmid-type, see also \cite{samson2000concentration} and \cite{paulin2015concentration}.

The remainder of this paper is organized as follows. In Section~\ref{Sec_Notation}, we give the notation used throughout the manuscript. Furthermore, we outline the basic concept of persistent homology. In Section~\ref{Sec_MainResults}, we describe the dependence structure assumed for our time series model and present our main results related to the time series case. In Section~\ref{Sec_ExtensionsToRandomFields}, we study the extension of our results to random fields. The proofs are contained in Section~\ref{Sec_TechnicalResults}; further deferred calculations are contained in \ref{Appendix1}.

\section{Notation}\label{Sec_Notation}
The purpose of this section is not to make the paper self-contained which is impossible. The aim is rather to allow the reader from other areas to become familiar with the vocabulary and to understand the basic concepts of topological data analysis. 

We begin with some general notation. We write $\N$ for the natural numbers starting at 1; if we include 0, we write $\N_0$. We write $\# A$ for the cardinality of a countable set $A$. We work on a separable Banach space $S$. We write $d$ for the metric which is obtained from the norm on $S$ and $\fS$ for the Borel-$\sigma$-field on $S$. $(S,\fS)$ is equipped with the measure $\mu$. The measure $\mu$ is non-atomic and $\sigma$-finite. Then we write $B(x,r) = \{y\in S: d(x,y)\le r\}$ for the closed $d$-ball around $x$. The diameter of a set $A\subseteq S$ is $\diam(A) = \sup\{ d(x,y): x,y\in S\}$. Let $A\in \cB(\R^p)$ and write $|A|$ for its $p$-dimensional Lebesgue measure as well as $A^{(\epsilon)} = \{x\in S: d(x,A)\le \epsilon\}$ for its $\epsilon$-offset.

Write $\otimes_{i=1}^\ell \mu = \mu^{\otimes \ell}$ for the $\ell$-fold product measure on the product space $(S^{ \ell},\fS^{\otimes \ell})$. The essential supremum of a real-valued function $f$ defined on $(S,\fS,\mu)$ is abbreviated by $\|f\|_{\infty,\mu}$. We write simply $\|f\|_\infty$ for the supremum norm of a continuous function on $\R^p$.

Let $\pspace$ be a probability space and let $(T,\fT), (U,\fU)$ be two state spaces. Consider two random variables $X\colon\Omega\to T$ and $Y\colon \Omega\to U$. Assume that $X$ admits a conditional distribution given $Y$. We write $\mM_{X|Y}\colon U\times \fT \to [0,1]$ for this distribution.

In order to abbreviate a subset of an ordered sample $(x_1,\ldots,x_n)$, say, we write $x_a^b$ for the subset $(x_a, x_{a+1}, \ldots, x_b)$, $1\le a\le b\le n$. Given a time series $X_1^n = (X_1,\ldots,X_n) \subseteq S$, we write $\mX_n = \{ X_1,\ldots, X_n\}$ for the associated point cloud which has no ordering.

Given a metric space $(E,d)$ and Radon measures $\nu,\nu_1,\nu_2,\ldots$, we say that $(\nu_n)_{n\in\N}$ converges vaguely to $\nu$ if
\[
			\int_E f\intd{\nu_n} \to \int_E f\intd{\nu}, \quad \forall\, f\in C_c(E),
\]
where $C_c(E)$ is the class of all continuous functions on $E$ with compact support. We indicate this writing $\nu_n \overset{v}{\to} \nu$.

We construct the filtration from the {\v C}ech or the Vietoris-Rips complex. If $\mX$ is a finite subset of $S$ and $r\ge 0$, these complexes are defined by
\begin{align*}
				\cC(\mX,r) &= \big\{ \sigma \subseteq \mX \ \big| \bigcap_{x\in \sigma} B(x,r)\neq \emptyset \big\} \text{ and }\\
				\cR(\mX,r) &= \{  \sigma \subseteq \mX \ \big | \diam(\sigma)\le r \}.
\end{align*}
In the following, the writing $K$ refers to both the {\v C}ech and the Vietoris-Rips complex. If we want additionally to precise the point cloud or the filtration parameter $r$, we write $K(\mX,r)$ or $K(r)$. The corresponding filtration is given by $\cK=\cK(\mX) = ( K(\mX,r): 0\le r< \infty )$. It is a direct consequence of the homogeneity of $d$ that for $\eta>0$ the complexes $K( \eta \mX,r)$ and $K(\mX,\eta^{-1} r)$ are combinatorially isomorphic.

 The dimension of a simplex $\sigma\in K$ is its cardinality minus 1. If $\sigma$ has dimension $q$, it is a $q$-simplex. Write $K_q$ for the set of $q$-simplices in a complex $K$. Moreover, for a measurable set $A\in \fS$ and a point cloud $\mX\subseteq S$, we write $K_q(\mX,r;A)$ for the number if $q$-simplices in $K(\mX,r)$ with at least one vertex in $A$.

We use the field $\mF_2$ to build the homology groups and the Betti numbers of a simplicial complex $K$.
Define for $q\in\N_0$ the space of $q$-chains $C_q(K)$ to be the free Abelian group generated by the $q$-simplices in $K$. So the elements of $C_q(K)$ are formal sums (``$q$-chains'') $c=\sum_{i} a_i \sigma_i$, $a_i\in\{0,1\}$, $\sigma_i\in K$ a $q$-simplex. The sum of two $q$-chains $c_1+c_2$ is their symmetric difference because the coefficients $a_i$ are in $\mF_2$.

The boundary operator $\partial_q$ relates $C_q(K)$ and $C_{q-1}(K)$ by mapping a $q$-simplex $\{x_0,x_1,\ldots,x_q\}$ to $\partial_q( \{x_0,x_1,\ldots,x_q\} ) \coloneqq \sum_{i=0}^q (-1)^{i+1} \{x_0,\ldots,x_{i-1},x_{i+1},\ldots,x_q\}$. For a general chain $c=\sum_i a_i \sigma_i \in C_q(K)$, the boundary operator is then $\partial_q(c) = \sum_i a_i \partial_q(\sigma_i) $. 

The boundary operator satisfies $\partial_{q} \circ \partial_{q+1} \equiv 0$ (``a boundary has no boundary''). This property enables the construction of homology groups of $K$. Let $Z_q(K) = \operatorname{ker}(\partial_q)$ be the subspace of  $C_q(K)$ consisting of the $q$-cycles, those elements whose boundary is 0 under $\partial_q$. Let $B_q(K)= \operatorname{im}(\partial_{q+1})$ be the subspace of $ C_q(K)$ that consists of the boundaries of elements in $C_{q+1}(K)$ (which lie in $C_q(K)$).

The homology groups are defined as $H_q(K) \coloneqq Z_q(K) / B_q(K)$, the cycles $Z_q$ modulo the boundaries $B_q$ in dimension $q$. Loosely speaking, the elements in $H_q(K)$ represent ``holes'' in the simplicial complex $K$. These are closed loops, voids or cavities, whose interior cannot be filled by other elements of the complex. Similarly as $Z_q(K)$ and $B_q(K)$, $H_q(K)$ is a vector space.

The $q$th Betti number of a simplicial complex $K$ is the dimension of $H_q(K)$, viz.,
\[
	\beta_q(K) = \dim ( Z_q(K)/B_q(K)).
\]
So, $\beta_q(K)$ is the number of $q$-dimensional holes in $K$. $H_q(K)$ and $\beta_q(K)$ provide topological information from a \textit{single} simplicial complex. Given a filtration $\cK = (K(r): 0\le r < \infty)$, the persistent homology provides more topological details. The natural inclusions $Z_q( K(r) ) \subseteq Z_q( K(s))$ and $B_q( K(r)) \subseteq B_q( K(s))$ for $r\le s$, provide the inclusion map
\[
	H_q(K(r)) \hookrightarrow H_q( K(s)),\quad x+B_q(K(r)) \mapsto x+B_q(K(s)).
\]
We define the persistent homology groups of the filtration $\cK = (K(r): 0\le r < \infty)$ by
\[
	H_q^{r,s}(\cK) = Z_q( K(r))/ ( B_q(K(s)) \cap Z_q(K(r)) ), \quad r\le s.
\]
Loosely speaking, nonzero elements in $H_q^{r,s}(\cK)$ represent topological features born before or at time $r$ and which persist until a time greater than $s$. The dimension of $H_q^{r,s}(\cK)$, i.e., the number of these features, is the persistent Betti number.
\begin{definition}[Persistent Betti number]\label{D:PersBetti}
Let $\cK$ be a filtration and let $0\le r\le s < \infty$. The persistent Betti number of dimension $q\in\N_0$ for the parameter pair $(r,s)$ is
\begin{align*}
	\beta_q^{r,s}(\cK) \coloneqq \dim H_q^{r,s}(\cK) 
	&=\dim ( Z_q( K(r)) ) - \dim ( B_q(K(s)) \cap Z_q(K(r)) ).
\end{align*}
\end{definition}
At this point there is an important difference between the \v Cech and Vietoris-Rips complex in the special case where we consider the Euclidean space $\R^p$. While in the \v Cech complex the homology degree is bounded by $p-1$, the Vietoris-Rips complex can have nontrivial cycles of every possible dimension (see also \cite{bobrowski2018topology}).

The $q$th persistence diagram summarizes the evolution of the homology groups; it is a multiset of points in $\Delta = \{ (b,d): 0\le b<d\le \infty \}$. Each point $(b,d)$ in the $q$th persistence diagram corresponds to a $q$-dimensional hole (feature) in the filtration $\cK$ which is born (appears for the first time) at time $b$ and dies (disappears in the filtration) at time $d$. The lifetime of this feature $d-b$ is called the persistence. $d=\infty$ means that the feature has an infinite lifetime.
Persistence diagrams exist given mild assumptions on the filtration, see \cite{chazal2016structure}. Also in the case of a random point cloud, e.g., an i.i.d.\ sample, the persistence diagram can inherit certain smoothness properties from the point cloud, see \cite{chazal2018density}.

Let $ \fD_q(\cK) = \{ (b_i,d_i)\in \Delta: i=1,\ldots, n_q \}$ be the $q$th persistence diagram given as a multiset of points. Then in the following we understand $\fD_q(\cK)$ as a counting measure on $\Delta$ defined as
\begin{align*}
		\xi_q(\cK) = \sum_{(b_i,d_i)\in \fD_q(\cK)} \delta_{(b_i,d_i)}.
\end{align*}
$\xi_q(\cK)$ is related to the $q$th persistent Betti number as follows
\begin{align*}
	&\xi_q(\cK)( [0,r]\times (s,\infty] ) =\beta^{r,s}_q( \cK ) .
\end{align*}

This means $\beta^{r,s}_q$ counts the number of $q$-dimensional features in the upper left rectangular area with vertex $(r,s)$ in the persistence diagram. So given $r << s$, the persistent Betti number $\beta^{r,s}_q$ represents the number of $q$-dimensional features with a high persistence.
It is clear that the values of the $q$th persistence diagram $\xi_q(\cK)$ also describe the persistent Betti function $\{\beta^{r,s}_q( \cK): 0\le r\le s < \infty\}$ completely.

\section{Persistent Betti numbers obtained from time series}\label{Sec_MainResults}
This section contains the main results of this paper. We derive an exponential inequality for persistent Betti numbers from a rather general class of stochastic processes, which also applies to functional data and random fields after a renumeration of the coordinates, we will see this below. For the special case of an $\R^p$-valued time series, we also give the large sample behavior of the expectation and study the vague convergence of the corresponding persistence diagram.

\subsection{The data generating process}
Consider a stationary process $X=( X_t: t\in \Z 	)$ defined on $\pspace$ and taking values in $S$. (A special case would be $\R^p$ equipped with the Borel-$\sigma$-field $\cB(\R^p)$ and the Lebesgue measure.)

The observations $X_t$ admit a density $\kappa$ w.r.t.\ $\mu$. Furthermore, the observations admit conditional densities as follows. The distribution of $X_t$ conditional on $X_1,\ldots,X_{t-1}$, $\cL(X_t | X_1,\ldots,X_{t-1})$, admits a density $f_{X_t \,|\, X_1,\ldots,X_{t-1} }$ for each $t\in\N$. Also $\cL( X_{v_1},\ldots, X_{v_\ell}  \,|\, X_t )$ admits a density $f_{X_{v_1},\ldots, X_{v_\ell}  \,|\, X_t}$ for all $t,\ell\in\N$ and all finite sets $\{v_1,\ldots,v_\ell\}\subseteq \N$, which do not contain $t$. Moreover, there is a $ f^*<\infty$ such that uniformly
\begin{align}
				 \kappa, f_{X_t \,|\, X_1,\ldots,X_{t-1} } \le f^* \text{ and } f_{ X_{v_1},\ldots, X_{v_\ell}  \,|\, X_t } \le f^*  \label{C:RegularityDensity} \tag{A1}
\end{align}
for all $t,\ell\in\N$ and sets $\{v_1,\ldots,v_\ell\}\subseteq \N$ which do contain $t$. These requirements are not restrictive and satisfied for a wide range of stochastic processes.

\subsection{Marton couplings as the concept of dependence}
We use the concept of Marton couplings to quantify the dependence within the observed data. These couplings were first defined in \cite{marton2003measure} and measure the strength of dependence within a collection of random variables by a mixing (or coupling) matrix.
\begin{definition}[Marton coupling]\label{Def:MartonCoupling}
Let $N\in\N$ and let $\Lambda_1,\ldots,\Lambda_N$ be Polish. Let $Z=(Z_1,\ldots,Z_N)$ be a vector of random variables taking values in $\Lambda = \Lambda_1\times\ldots\times\Lambda_N$. A Marton coupling of $Z$ is a set of couplings $(	 Z^{(z_1,\ldots,z_i,z'_i)}, Z'^{(z_1,\ldots,z_i,z'_i)} $), for every $i \in \{1,\ldots,N\}$ and every $z_1\in\Lambda_1,\ldots, z_i, z'_i \in\Lambda_i$, which satisfies the conditions
			\begin{align*}
		(i)	\qquad			& Z_j^{(z_1,\ldots,z_i,z'_i)} = z_j \text{ for all } j \in\{1,\ldots,i\}, \\
						& {Z'}_j^{(z_1,\ldots,z_i,z'_i)} = z_j \text{ for all } j \in \{1,\ldots,i-1\} \text{ and } {Z'}_i^{(z_1,\ldots,z_i,z'_i)} = z'_i. \\
		(ii)  \qquad			&  \big(	Z_{i+1}^{(z_1,\ldots,z_i,z'_i)},\ldots,Z_{N}^{(z_1,\ldots,z_i,z'_i)}	\big) \sim \cL\big(	Z_{i+1}.\ldots,Z_N \,|\, Z_1=z_1,\ldots,Z_{i-1}=z_{i-1},Z_i=z_i	\big) , \\
		&  \big(	{Z'}_{i+1}^{(z_1,\ldots,z_i,z'_i)},\ldots,{Z'}_{N}^{(z_1,\ldots,z_i,z'_i)}	\big) \sim \cL\big(	Z_{i+1}.\ldots,Z_N \,|\, Z_1=z_1,\ldots,Z_{i-1}=z_{i-1},Z_i=z'_i	\big). \\
		(iii)	\qquad			& \text{If } z_i=z'_i, \text{ then }  Z^{(z_1,\ldots,z_i,z'_i)}=  {Z'}^{(z_1,\ldots,z_i,z'_i)}.
			\end{align*}
Write $\mathbbm{M}_{Z_i|(Z_1,\ldots,Z_{i-1})} $ for the conditional distribution of $Z_i$ given $(Z_1,\ldots,Z_{i-1})$ for $1\le i\le N$. Construct a measure $\mu_i$ on the product space $\Lambda_1\times\ldots\times \Lambda_{i-1}\times \Lambda_i\times\Lambda_i$, which consists of the joint distribution of $(Z_1,\ldots,Z_{i-1})$ and the product measure $\mathbbm{M}_{Z_i|(Z_1,\ldots,Z_{i-1})} \otimes \mathbbm{M}_{Z_i|(Z_1,\ldots,Z_{i-1})} $ for a Borel set of $\Lambda_1\times \ldots \Lambda_{i-1} \times \Lambda_i \times \Lambda_i$ as follows:
\begin{align}\begin{split}\label{Def:MuIMeasure}
		\mu_i(A ) &= \int_{ \Lambda_1\times\ldots\times \Lambda_{i-1} }  \p_{(Z_1,\ldots,Z_{i-1}) }(\intd{(z_1,\ldots,z_{i-1})} ) \int_{\Lambda_i } \mathbbm{M}_{Z_i|(Z_1,\ldots,Z_{i-1})} \left( (z_1,\ldots,z_{i-1}), \intd{z_i} \right)\\
		&\qquad\qquad  \int_{\Lambda_i}\mathbbm{M}_{Z_i|(Z_1,\ldots,Z_{i-1})} \left( (z_1,\ldots,z_{i-1}), \intd{z'_i} \right) \1{A} (z_1,\ldots,z_{i-1},z_i,z'_i).
\end{split}\end{align}
Then, we define the mixing matrix $\Gamma \coloneqq (\Gamma_{i,j})_{1\le i,j \le N}$ for a Marton coupling of $Z$ as an upper diagonal matrix with $\Gamma_{i,i} =1$ and
$$
			\Gamma_{j,i}=0, \quad \Gamma_{i,j} = \esssup_{ } \p\left( Z_j^{(z_1,\ldots,z_{i-1}, z_i,z'_i)} \neq {Z'}_j^{(z_1,\ldots,z_{i-1}, z_i,z'_i)} \right), \quad 1\le i < j \le N,
$$
where we compute the essential supremum w.r.t.\ the measure $\mu_i$. Note that for $1\le i<j\le N$ each entry in the mixing matrix is bounded above by
\[
	\Gamma_{i,j} \le \sup_{}  \p\left( Z_j^{(z_1,\ldots,z_{i-1}, z_i,z'_i)} \neq {Z'}_j^{(z_1,\ldots,z_{i-1}, z_i,z'_i)} \right)
	\]
	where the supremum is taken over all $(z_1,\ldots,z_{i-1}, z_i,z'_i)\in\Lambda_1\times\ldots\times\Lambda_i\times\Lambda_i$.
\end{definition}

We return to the data generating process $X$. Write $\Gamma^{(n)}$ for the mixing matrix of the sample $X_1,\ldots,X_n$. As $X$ is stationary, $\Gamma^{(n)}_{i,j} = \Gamma^{(n)}_{i+k,j+k}$ (as long as all indices are between 1 and $n$). Consequently, $\Gamma^{(n)}_{i,j} = \Gamma^{(n)}_{n-j+1,n-i+1}$ for the choice $k = n-j-i+1$. So the summation over all elements in line $i$ is equivalent to the summation over all elements in column $n-i+1$ (and vice versa), viz.,
\begin{align*}
				\sum_{j=1}^n \Gamma^{(n)}_{i,j}  = \sum_{j=i}^n \Gamma^{(n)}_{i,j}  = \sum_{j=1}^{n-i+1} \Gamma^{(n)}_{j,n-i+1} = \sum_{j=1}^{n} \Gamma^{(n)}_{j,n-i+1}.
\end{align*}
In particular, the maximum absolute column sum $\| \Gamma^{(n)} \|_1 $ equals the maximum absolute row sum $\|\Gamma^{(n)}\|_\infty$. In what follows, we assume that the coefficients of the mixing matrix are at least summable in the sense that
\begin{align}
			\gamma_\infty \coloneqq \sup_{n\in\N} \| \Gamma^{(n)} \|_\infty  < \infty. \label{C:MixingMatrix} \tag{A2}
\end{align}
Consider the spectral norm $\|\Gamma^{(n)}\|$ of the mixing matrix $X$ induced by the Euclidean norm $\|\cdot\|$ on $\R^N$. Using $\| \Gamma^{(n)} \|^2 \le \| \Gamma^{(n)} \|_1 \| \Gamma^{(n)} \|_\infty$, implies that $\Gamma^{(n)}$ is also uniformly bounded in the spectral norm over all $n\in\N$, i.e., $\sup_{n\in\N} \| \Gamma^{(n)} \|< \infty$.

The condition on the mixing matrix in \eqref{C:MixingMatrix} is satisfied for a wide range of stochastic processes. Consider for instance, so-called delay embeddings for time series.

\begin{example}[Delay embeddings from Markov chains]\label{MarkovChainDelayEmbedding}
Let $Z$ be a stationary, uniformly geometrically ergodic Markov chain in a Polish space $\cE$ whose marginal distribution and transition kernel both admit a strictly positive density w.r.t.\ a reference measure $\mu$. Construct a process $X$ from $Z$ via a delay embedding, that is, $X_t = (Z_{t}, Z_{t-\tau_1},\ldots, Z_{t-\tau_{m-1}})\in\cE^m$, where $\tau_1< \ldots < \tau_{m-1}$ are natural numbers. We show that this process $X$ satisfies \eqref{C:MixingMatrix}. We construct a Marton coupling $(X^{(x_1,\ldots,x_i,x'_i)}, X'^{(x_1,\ldots,x_i,x'_i)}	)$, for every $i \in \{ 1,\ldots,n \}$ and every $x_1,\ldots,x_{i-1}, x_i, x'_i \in \cE^m$ with Goldstein's maximal coupling (Proposition~\ref{Prop:Goldstein}). 

For every $i$ and all states, Goldstein's maximal coupling yields two coupled random variables $ X^{(x_1,\ldots,x_i,x'_i)}$, $X'^{(x_1,\ldots,x_i,x'_i)}$ such that (i), (ii) and (iii) from Definition~\ref{Def:MartonCoupling} are satisfied. By Proposition~\ref{Prop:Goldstein}, the marginals of each coupling satisfy
\begin{align}
			&\p\left(	X_j^{(x_1,\ldots,x_i,x'_i)} \neq X_j'^{(x_1,\ldots,x_i,x'_i)}		\right) \nonumber \\
			&\le d_{TV}\left(	\cL((X^{(x_1,\ldots,x_i,x'_i)})_j^n ), \cL((X'^{(x_1,\ldots,x_i,x'_i)})_j^n )		\right)  \nonumber \\
			&= d_{TV} ( \cL(X_j^n|X_1^{i-1}=x_1^{i-1},X_i=x_i), \cL(X_j^n|X_1^{i-1}=x_1^{i-1},X_i=x'_i) ).\label{E:DelayEmbedding1}
\end{align} 
Note that the essential supremum of the left-hand side w.r.t.\ $\mu_i$ equals the coefficient $\Gamma^{(n)}_{i,j}$. Thus, we can easily bound above the norm of the mixing matrix $\Gamma^{(n)}$ with the properties of the Markov chain $Z$. For simplicity, we use $\Gamma^{(n)}_{i,j} \le 1$ for $0\le j-i\le \tau_{m-1}$ and only consider the asymptotic properties for $j-i>\tau_{m-1}$. Set $x_i=(z_i,\ldots,z_{i-\tau_{m-1}})'$. We can derive from the Markov property of $Z$ that
\begin{align*}
			\p\left( X_j^n \in A | X_1^i = x_1^i \right) &= 
			\p\left\{
			\left(	
			\begin{pmatrix}
			Z_j\\ Z_{j-\tau_1}\\ \vdots \\ Z_{j-\tau_{m-1} }
			\end{pmatrix}
			,\ldots,
			\begin{pmatrix}
			Z_n\\ Z_{n-\tau_1}\\ \vdots \\ Z_{n-\tau_{m-1} }
			\end{pmatrix}
			\right) \in A
			\Biggl|
			Z_i=z_i \right\}.
\end{align*}
Next, we use the Markov property to see that the total variation distance in \eqref{E:DelayEmbedding1} is determined by the observation $X_j$ because $j$ is closest to $i$. Consequently, if $j-\tau_{m-1}-i\ge 0$, \eqref{E:DelayEmbedding1} equals 
\begin{align}\label{E:DelayEmbedding2}
		d_{TV}\left( \cL( Z_{j-\tau_{m-1}} |Z_i=z_i), \cL(Z_{j-\tau_{m-1}} |Z_i =z'_i 	) \right).
\end{align}
By assumption, $Z$ is uniformly geometrically ergodic. Hence, there are $R\ge 1$ and $\rho\in (0,1)$ such that uniformly in $z$ and for all $t\in\N$
\[
	d_{TV}\Big( \cL( Z_{t} |Z_0=z), \cL(Z_{t}	) \Big) \le R \rho^t.
\]
So the quantity in \eqref{E:DelayEmbedding2} is at most $1\wedge (2R \rho^{j-\tau_{m-1}-i})$. In particular, we have for a row of the mixing matrix of $X_1,\ldots,X_n$ the following bound, which implies \eqref{C:MixingMatrix}:
\[
		\Gamma^{(n)}_{i,\cdot} \le (1,\ldots,1,1, 1 \wedge (2R \rho,1 \wedge (2R \rho^2),\ldots,1\wedge(2R \rho^{n-1-\tau_{m-1}} ) ).
	\]
Consequently, $\gamma_\infty \le (\tau_{m-1}+1 - 2R ) + 2R / (1- \rho)$.

 The mixing time of a (uniformly ergodic) Markov chain is defined by
\[
	t_{mix} = \min\Big( t: d_{TV}( \cL( Z_{t} |Z_0=z), \cL(Z_{t} )  ) \le \frac{1}{4}	\Big).
\]
Hence, using the Markov property, we can also give an upper bound on \eqref{E:DelayEmbedding2} in terms of the mixing time by simply writing $j-\tau_{m-1}-i = k t_{mix} + r$ for $k\in\N_0$ and $r\in \{0,\ldots,t_{mix}-1\}$. Then \eqref{E:DelayEmbedding2} is at most 1 if $k=0$ or if $j-\tau_{m-1}-i < 0$ and $(2 \frac{1}{4})^k$ if $k\ge 1$. Consequently, one obtains the following upper bound for $\gamma_\infty$
\[
	\gamma_{\infty} \le \tau_{m-1}  + t_{mix} \cdot 1 + t_{mix} \cdot \frac{1}{2} + t_{mix} \cdot \frac{1}{4} + \ldots =\tau_{m-1} + 2 t_{mix}.
\]
\end{example}

\subsection{Covering the state space}
As we consider general state spaces $S$, we work with the following covering condition, which is satisfied in many examples.

\begin{condition}[Covering condition]\label{C:CoveringCondition}
The state space $(S,\fS)$ is precompact. Write $\sN=\sN(r, S, d)$ for the $r$-covering number of $S$ w.r.t.\ $d$, i.e., for each $r>0$, $S$ admits a covering $\{ B(w_j,r): 1\le j\le \sN \}$ with closed balls w.r.t.\ the metric $d$ of radius $r$ located at positions $w_j$.

Moreover for all $r\in\R_+$ there is a sequence of scaling factors $(\eta_n:n\in\N)\subseteq\R_+$, $\eta_n\to\infty$, such that
\begin{align}
		&  n  \sup_{w\in S} \mu( B(w,2\eta_n^{-1} r ) )  < C_r (\log n)^{\alpha_1}, \label{C:Covering} \tag{A3} \\
		&\log \sN(\eta_n^{-1} r, S, d) \le C_r (\log n)^{\alpha_2} \label{C:Covering2} \tag{A4}
\end{align}
for some $C_r,\alpha_1, \alpha_2 \in\R$.
\end{condition}
Some discussion on the covering condition is needed. \eqref{C:Covering2} is clearly needed to bound above the complexity of the underlying metric space $(S,d)$. Condition \eqref{C:Covering} is more delicate as it regulates the ratio between the number of points $n$ and the $\mu$-volume of the $d$-ball. Many spaces satisfy this condition. We give here two examples, the finite-dimensional case, i.e., $\R^p$, and the functional case.

\begin{example}[Coverings for finite-dimensional spaces]
Consider the unit cube $S= [0,1]^p$ which is endowed with the $\infty$-norm and the Lebesgue measure $|\cdot|$. For each $r>0$, $[0,1]^p$ can be covered with disjoint cubes of side length at most $2r$, i.e., balls w.r.t.\ the $\infty$-norm $B(w_j,r) = w_j + [-r,r]^p$. In that case, one finds with geometric arguments that a ball of radius $r$ at some position $w$ can be covered with at most $2^p$ balls of radius $r$ at fixed positions $w_j$. 

In this Euclidean setting, three regimes are classically studied for random geometric complexes. In the subcritical regime $\eta_n^{-1} n^{1/p} \rightarrow 0$, i.e., the scaling factors grow faster than $n^{1/p}$. In the critical regime, this growth is balanced, so that $\eta_n^{-1} n^{1/p} \rightarrow \eta^{-1}\in\R_+$. Moreover, $\eta_n^{-1} n^{1/p} \rightarrow \infty$ in the supercritical regime.

We study the situation for a point cloud of $n$ points $\mX_n \subseteq [0,1]^p$ obtained from a stationary time series $X_1,\ldots,X_n$, whose marginals admit a density $\kappa$ w.r.t.\ the Lebesgue measure. In the subcritical regime, the points of the rescaled point cloud $\eta_n\mX_n$ tend to become more and more isolated as the number of points per volume tends to zero. In the critical regime, the number of points per volume from $\eta_n\mX_n$ tends to a constant. In the supercritical regime, the points from the point cloud $\eta_n \mX_n$ lie increasingly dense.

Clearly, scaling factors which achieve the thermodynamic regime (e.g.\ $\eta_n = n^{1/p}$) satisfy the condition from \eqref{C:Covering} because $n |B(x,\eta_n^{-1} s)| \propto s^p$. The covering number $\sN( \eta_n^{-1} s, [0,1]^p, \|\cdot\|_{\infty} )$ is proportional to $n$ in this case.

Moreover, note that we can also allow for a slower increase in $\eta_n$ which then yields a supercritical regime. For instance, $\eta_n \propto n^{1/p} (\log n)^{-\alpha}$ still satisfies \eqref{C:Covering} (and also \eqref{C:Covering2}) for each $\alpha>0$.

Scaling factors which achieve a subcritical regime satisfy \eqref{C:Covering}, however, \eqref{C:Covering2} restricts the growth rate from above; for instance, any polynomial rate $\eta_n = n^{\alpha}$ is allowed for \eqref{C:Covering2}.
\end{example}

\begin{example}[Coverings of functional spaces] 
Let $\alpha>0$. We study the class of all functions $x$ on the unit interval that posses uniformly bounded derivatives on $(0,1)$ up to order $\underline{\alpha}$ (the greatest integer strictly smaller than $\alpha$) and whose highest derivatives are H{\" o}lder continuous of order $\alpha-\underline{\alpha}$. Write $D^k x$ for the $k$th derivative of a function $x$. Define
\[
		\norm{x}_\alpha = \max_{k\le \underline\alpha} \sup_{t\in[0,1]} |D^k x(t)| +  \sup_{ \substack{s,t\in (0,1),\\ s\neq t } } \frac{ |D^{\underline\alpha} x(t) - D^{\underline\alpha} x(s)|}{|s-t|^{\alpha-\underline\alpha}}.
\]
Write $C^\alpha_M([0,1])$ for the set of continuous real-valued functions on $[0,1]$ with $\norm{x}_\alpha \le M$. We write $\|\cdot\|_{\infty}$ for the supremum-norm of a real-valued function on $[0,1]$ and denote the corresponding $r$-neighborhood of $x$ by $B(x,r)$ (other norms are also possible). Then the covering number of $C^\alpha_1 ([0,1])$ w.r.t.\ $\|\cdot\|_\infty$ satisfies by Theorem~2.7.1 in \cite{vanweak}
\[
		\log \sN(\epsilon,C^\alpha_1 ([0,1]), \norm{\cdot}_\infty) \le c \epsilon^{-1/\alpha},
\]
for a certain constant $c\in\R_+$ which is independent of $\epsilon>0$. 

For each centered Gaussian measure $\mu$ on a real separable Banach space $E$, there is a unique Hilbert space $H_\mu \subseteq E$ such that $\mu$ is determined by considering $(E,H_\mu)$ as an abstract Wiener space (Cameron-Martin space, see \cite{gross1970abstract}). Then $H_\mu$ is dense in $E$ and we have for $x\in H_\mu$
\begin{align}\label{E:SmallBallProb2}
		C^{-1}_x \mu( B(0,s)) \le \mu( B(x,s)) \le \mu(B(0,s)),
\end{align}
where $C_x \ge 1$, see \cite{li2001gaussian} Theorem 3.1. 
In the following, consider the centered Gaussian measure on the infinitely often differentiable functions determined by the covariance function $k(s,t) = \exp(-(s-t)^2 / (2\ell^2))$ for a given characteristic length-scale $\ell>0$. In this case, it follows from \cite{li2001gaussian} Theorem 4.1 and Theorem 4.5 as well as some calculations that
\begin{align}\label{E:SmallBallProb3}
	\exp( - K_1/(\ell s) ) \le \mu( B(0,s) ) \le \exp( - K_2/(\ell s) )
\end{align}
for all $s$ sufficiently small for certain $K_1,K_2\in\R_+$.

For instance, define the scaling factor by $\eta_n = v (\log n)^{\beta}$ for $\beta , v> 0$. Then on the one hand, we obtain $\log \sN(\eta_n^{-1}s) \le c (v/s)^{1/\alpha} (\log n)^{\beta/\alpha}$ which shows \eqref{C:Covering2}. On the other hand, relying on \eqref{E:SmallBallProb3}
\begin{align*}
	C_x^{-1} \exp\Big( \log n - \frac{K_1 v}{2\ell s} (\log n)^\beta \Big) &\le n \mu( B(x,2\eta_n^{-1} s) ) \le \exp\Big( \log n - \frac{K_2 v}{2\ell s} (\log n)^\beta \Big).
\end{align*}
If $\beta>1$, \eqref{C:Covering} is satisfied and we have a functional subcritical regime in that $n \mu( B(x,2\eta_n^{-1}s) )\to 0$ ($n\to\infty$).
\end{example}

\subsection{A concentration inequality for persistent Betti numbers from dependent data}

We come to the first main result in this article which is an abstract exponential inequality for a certain class of functionals defined on the point clouds $\mX_n$, $n\in\N$.
\begin{theorem}\label{T:AbstractExpIneq}
Let the stochastic process $X$ satisfy assumption \eqref{C:RegularityDensity} and \eqref{C:MixingMatrix}. Moreover, let \eqref{C:Covering} and \eqref{C:Covering2} from Condition~\ref{C:CoveringCondition} be satisfied. Let $H$ be a functional defined on all finite subsets of $S$. Set $H_n( \{x_1,\ldots,x_n\} ) = H(\eta_n \{x_1,\ldots,x_n\} )$. The functional satisfies
\begin{itemize}
	\item [(1)] Universal bound. There are $c_1,q\in\R_+$ such that for all $n\in\N$, $1\le i\le n$ and for all $x_1,\ldots,x_n,y\in S$ 
	\begin{align}\label{E:UniversalBound}
		| H_n( \{x_1,\ldots,x_n\} ) - H_n( \{x_1,\ldots,x_{i-1}, y, x_{i+1},\ldots,x_n \} ) | \le c_1 n^q.
	\end{align}
	\item [(2)] Exchange-one cost are local. There are $c_2, r, \wt q\in\R_+$ such that for all $n\in\N$, $1\le i\le n$ and for all $x_1,\ldots,x_n,y\in S$ 
	\begin{align}\begin{split}\label{E:ExchOneCost}
		&| H_n( \{x_1,\ldots,x_n\} ) - H_n( \{x_1,\ldots,x_{i-1}, y, x_{i+1},\ldots,x_n \} ) |  \\
		& \le c_2 |\{y,x_1,\ldots,x_n\} \cap B(x_i,\eta_n^{-1} r) |^{\wt q} + c_2 | \{y,x_1,\ldots,x_n\} \cap B(y,\eta_n^{-1} r) |^{\wt q}.
	\end{split}\end{align}
\end{itemize}
Set $\gamma = (2a-1)/(2\wt q +1)$ and let $a>1/2$. Then for all $n\in\N$ and for all $t>0$
\begin{align*}
	&\p( | H_n(\mX_n) - \E{ H_n(\mX_n)} | > n^a t) \\
	 &\le 2 \exp\left(	-  \frac{n^{\gamma} t^2 }{ 4^{\wt q}2 (16 c_2 \gamma_\infty )^2} \right) + \frac{2 c_1 e n^{2q+1 -\gamma \wt q-a}}{c_2 2^{\wt q}\gamma_\infty} \Big( \frac{1}{n^{q-a}} + \frac{2c_1}{t} \Big) \\
				&\quad  \times \exp\Big\{ - \big[ n^\gamma - \log\sN(\eta_n^{-1}r) - f^*(e-1) \ n \sup_{w\in S} \mu( B(w,2\eta_n^{-1}r)) 		\big]		\Big\}.
\end{align*}
\end{theorem}

The persistent Betti function satisfies the conditions of Theorem~\ref{T:AbstractExpIneq}, this follows from the Geometric Lemma (Lemma~\ref{L:GeometricLemma}). More precisely, we have a result as follows.
\begin{theorem}\label{T:ExponentialInequality}
Let the regularity conditions from Theorem~\ref{T:AbstractExpIneq} be satisfied. Then for each $q\in\N_0$, for each $(r,s)$, $r\le s$, for $ a>1/2$ and $\gamma = (2a-1)/(2q+3)$ such that the persistent Betti number from Definition~\ref{D:PersBetti} satisfies for each $t>0$
\begin{align*}
		&\p\left(	|\beta_q^{r,s}(\cK(\eta_n \mX_n)) - \E{\beta_q^{r,s}(\cK(\eta_n \mX_n))} | \ge n^a t	\right)\\
&\le 2 \exp\left(	-  \frac{n^\gamma t^2 }{2 (2^{q+1} 32 \gamma_\infty)^2} \right) + \frac{e n^{(q+1)(1-\gamma)+q-a}}{2^{q+1}\gamma_\infty} \Big( \frac{1}{n^{q-a}} + \frac{2}{t} \Big) \\
				&\quad  \times \exp\Big\{ - \big[ n^\gamma - \log\sN(2\eta_n^{-1}s,S,d) - f^*(e-1) \ n \sup_{w\in S} \mu( B(w,4\eta_n^{-1}s)) 		\big]		\Big\}.
\end{align*}
Hence, there are constants $C_1,C_2,C_3\in\R_+$ such that for each $t>0$ and $n\in\N$ 
\begin{align*}
		&\p\left(	|\beta_q^{r,s}(\cK(\eta_n \mX_n)) - \E{\beta_q^{r,s}(\cK(\eta_n \mX_n))} | \ge n^a t	\right)\\
		&\le 2 \exp\left(	-  C_1 n^\gamma t^2 \right) + C_2 t^{-1} n^{(q+1)(1-\gamma)+q-a}  \exp( - C_3 n^\gamma ).
\end{align*}
In particular, $n^{-1} (\beta_q^{r,s}(\cK(\eta_n \mX_n)) - \mathbb{E}[\beta_q^{r,s}(\cK(\eta_n \mX_n)) ] ) \rightarrow 0$ $a.s.$
\end{theorem}

\cite{yogeshwaran2017random} obtain (among other results) an exponential inequality for Betti numbers computed from an $n$-binomial process, whose marginals have a continuous and compactly supported density on $\R^d$.
Their result, in particular, their rate is very similar to our result.

The abstract concentration inequality in Theorem~\ref{T:AbstractExpIneq} can be considered as a generalization of McDiarmid's inequality for functionals of dependent random vectors whose martingale differences are not necessarily bounded. Technically it relies on the Marton coupling and an abstract concentration inequality of \cite{chalker1999size} for martingale differences as well as on the observation that point cloud data from the data generating process tends to evenly spread across the state space $S$. So exchanging one point in the argument of the functional as in \eqref{E:ExchOneCost} tends to have a much smaller impact than the worst case bound in \eqref{E:UniversalBound}.

Many important functionals in stochastic geometry do not possess a deterministic local exchange-one cost function as required in Theorem~\ref{T:AbstractExpIneq}. Instead these functionals are often stabilizing (at an exponential rate). We refer to \cite{penrose2001central} and \cite{lachieze2019normal} for an introduction and examples of stabilizing functionals. Loosely speaking stabilization implies that the exchange-one cost function is determined by the points in a ``small'' $r$-neighborhood of the exchanged points with a high probability. So the abstract result in Theorem~\ref{T:AbstractExpIneq} will also be relevant for such stabilizing functionals because usually these can be truncated in such a way that the error is negligible for large sample sizes and that the exchange-one cost of the truncated functional actually satisfies \eqref{E:ExchOneCost}. Of course, this remark is an outline of how to apply Theorem~\ref{T:AbstractExpIneq} to such functionals and we leave it to future research to prove rigorously this claim.

\subsection{Convergence results in Euclidean space}
It remains to show that the normalized expectation of the persistent Betti numbers converges to a limit. Here we restrict our considerations to point cloud data on $\R^p$ because of the following reason: In order to obtain limit theorems for persistent Betti numbers in the critical regime from dependent data realized on a general measure space such as a manifold or a function space, a possible way would be to first derive the limit of $n^{-1} \mathbb{E}[\beta_q^{r,s}(\cK(\eta_n \cP_n))]$ for a certain underlying Poisson process $\cP_n$ on this space (see \cite{last2017lectures} for the notion of a Poisson process on a general space). In this step, the topological properties of the underlying space are of course crucial. Second, one needs to apply a de-Poissonization argument to obtain a limit for the binomial process which treats the situation for an i.i.d.\ sample. Finally, as we will see from the applied techniques in the proofs, the nearly additive properties of persistent Betti numbers in the critical regime enable us to use certain continuity results and then allow us to conclude the case for dependent data. This entire procedure is quite comprehensive. So far, to the best of our knowledge, these extensions have only been considered for manifolds (\cite{goel2018strong}) in the literature. For this reason we have decided to limit our considerations to $\R^p$-valued data.

We study two kind of processes $X=(X_t:t\in\Z) \subseteq [0,1]^p$ in the critical regime, namely, (1) processes which can be coupled to a process with a discrete state space and (2) Markov chains of finite order.

First consider a process $X$ which is obtained from a stationary discrete process $Z = (Z_t: t\in\Z)$ as follows. Let $\kappa$ be a blocked density w.r.t.\ the Lebesgue measure on $[0,1]^p$, i.e., there is an $m\in\N$ such that
\begin{align}
		\kappa = \sum_{i=1}^{m^p} \alpha_i \1{A_i},\quad \alpha_i > 0,\quad \sum_{i=1}^{m^p} \alpha_i |A_i| =1 \label{E:BlockedDensity} \tag{A5}
	\end{align}
and where the subcubes $A_i$ partition $[0,1]^p$. Note that the $A_i$ may have different volumes. 

We assume that the process $Z$ admits a Marton coupling which satisfies \eqref{C:MixingMatrix} and takes values in a finite set $S=(s_1,\ldots,s_{m^d})$, $s_i\neq s_j$ for $1\le i\neq j \le m^p$, such that $\p( Z_t = s_i ) = \alpha_i |A_i|$. Define $X_t$ with the help of $Z_t$ by
\begin{align}
			X_t = \sum_{i=1}^{m^p} Y_{t,i} \1{ Z_t = s_i}, \label{E:BlockedDensity2} \tag{A6}
\end{align}
where the $Y_{t,i}$ are independent and uniformly distributed on $A_i$ for $t\in\Z$ and $1\le i\le m^d$. Then if $B\subseteq A_i$, $\p( X_t \in B ) = (|B| / |A_i|) \, (\alpha_i |A_i|) = \alpha_i |B|$. Hence, each $X_t$ admits a marginal density $\kappa$.

Then the conditional distribution of the process $X$ works as follows. In the first step and conditional on the past $(X_1,\ldots,X_{t-1})$, we choose a subcube $A_i$, according to
\[
		\p(X_t\in A_i | X_1 \in A_{j_1} ,\ldots,X_{t-1}\in A_{j_{t-1}} ) =\p( Z_t = s_i | Z_1 = s_{j_1} ,\ldots,Z_{t-1} = s_{j_{t-1}} ).
			\]
In the second step, we choose at random a point in the subcube $A_i$ as the realization of $X_t$. 

Consequently, $X=(X_t:t\in\Z)$ admits a Marton coupling which satisfies \eqref{C:MixingMatrix}. The conditional distribution of $X$ is invariant in the sense that
\begin{align}\label{E:InvarianceOnCells}
			\p( X_t^{t+\ell} \in A | X_{t-s}^{t-1} = x_{t-s}^{t-1} ) = \p( X_t^{t+\ell} \in A | X_{t-s}^{t-1} = y_{t-s}^{t-1} ),
\end{align}
for all $x_i,y_i \in A_{i}$, $i=t-s,\ldots,t-1$, $\ell,s\in \N$, $t-s\ge 1$. If we can only observe the process $X$, then we can think of $Z$ as a hidden process. 
We have the following theorem.

\begin{theorem}\label{T:ConvergenceExpectationDiscrete}
Let $X=( X_t: t\in\Z )$ be a $[0,1]^p$-valued process which admits a Marton coupling that satisfies \eqref{C:MixingMatrix}. Each $X_t$ has a marginal density $\kappa$ as in \eqref{E:BlockedDensity} and \eqref{E:BlockedDensity2} such that $0<\inf \kappa \le \sup \kappa < \infty$. Then for each $0\le q\le p-1$
\begin{align}
			\lim_{n\rightarrow \infty} n^{-1} \E{ \beta^{r,s}_q( \cK( n^{1/p} \mX_n) ) } &= \E{ \hat{b}_q(\kappa(X_t)^{1/p}(r,s)) },	\label{E:LimitDiscrete} \\
			\lim_{n\rightarrow \infty}  n^{-1} \beta^{r,s}_q( \cK( n^{1/p} \mX_n) )  &= \E{ \hat{b}_q(\kappa(X_t)^{1/p}(r,s)) } \quad  a.s., \label{E:LimitDiscreteB}
\end{align}
where $\hat{b}_q(r,s)$ is the limit of $n^{-1} \E{ \beta^{r,s}_q(\cK(n^{1/p} \mX^*_n)) }$ for an $n$-binomial process $\mX^*_n$ with unit density on $[0,1]^p$.
\end{theorem}

So the expectation of the persistent Betti number obtained from this kind of time series has the same limit properties as the corresponding binomial process.

We extend Theorem~\ref{T:ConvergenceExpectationDiscrete} to general marginal density functions $\kappa\colon [0,1]^p \to \R_+$ which can be approximated by blocked density functions $\kappa_\epsilon$. To this end, we restrict ourselves to the case of uniformly ergodic Markov chains of order $m$, viz., $X=(X_t:t\in\Z)$ is a stationary process such that $\cL(X_t|X_u: u< t ) = \cL( X_t| X_{t-1},\ldots,X_{t-m})$, for some $m\in\N$. For such a Markov chain all transition probabilities are determined by the joint density $g$ of $X_1,\ldots,X_{m+1}$ which is assumed to be continuous and strictly positive on $[0,1]^{(m+1)p}$ in that $\inf \{ g(z): z\in [0,1]^{(m+1)p} \} > 0$. It is known that this kind of aperiodic restriction ensures the Markov chain $X$ to be uniformly geometrically ergodic, see also \cite{meyn2012markov}~Theorem 16.0.2.

Furthermore, the limit on the right-hand side in \eqref{E:LimitDiscrete} is continuous: Indeed, \cite{divol2018persistence} show that the limit
\begin{align}\label{E:LimitContinuity}
			\lim_{\epsilon\to 0} \E{  \hat{b}_q(\kappa_\epsilon(Y_\epsilon)^{1/p}(r,s)) } = \E{  \hat{b}_q(\kappa(Y)^{1/p}(r,s)) }
\end{align}
exists, where $\kappa_\epsilon$ are blocked density functions on $[0,1]^p$ (from a regular grid) as in \eqref{E:BlockedDensity} which converge to $\kappa$ in the $\|\cdot\|_\infty$-norm and where the $Y_\epsilon$ (resp. $Y$) have density $\kappa_\epsilon$ (resp. $\kappa$).

For this kind of Markov chains $X$ we obtain from the previous Theorem~\ref{T:ConvergenceExpectationDiscrete} the following result.
\begin{theorem}\label{T:ConvergenceExpectationContinuous}
Let $X$ be a homogeneous Markov chain of order $m$ taking values in $[0,1]^p$ such that the joint density $g$ of $X_1,\ldots,X_{m+1}$ is continuous and satisfies $\inf \{ g(z): z\in [0,1]^{(m+1)p} \} > 0$. The $X_t$ have marginal density $\kappa$. Then for each $q=0,\ldots,p-1$ the convergence results from \eqref{E:LimitDiscrete} and \eqref{E:LimitDiscreteB} in Theorem~\ref{T:ConvergenceExpectationDiscrete} are valid.
\end{theorem}

Consequently, we obtain also for this natural generalization of the binomial process the well-known limit. The generalization to arbitrary stationary processes $X$ which admit a Marton coupling is rather elaborate and complex. Actually, when following the current scheme of the proof, one first has to assume that this process $X$ can be coupled to a discrete process $\wt{X}$ which approximates $X$ sufficiently closely in terms of the conditional distribution functions. This would mainly result in a complex notation. For this reason, we have limited our considerations to processes whose conditional distributions only depend on $m$ lags of its past, this is sufficient for many applications and also serves as an approximation to the general case.

We conclude with an immediate result which follows from the Theorem~\ref{T:ConvergenceExpectationDiscrete} and Theorem~\ref{T:ConvergenceExpectationContinuous} and the work of \cite{hiraoka2018limit} concerning the vague convergence of persistence diagrams.
\begin{corollary}[Vague convergence of persistence diagrams obtained from dependent data] \label{C:VagueConvergence}
Let the assumptions of Theorem~\ref{T:ConvergenceExpectationDiscrete} or Theorem~\ref{T:ConvergenceExpectationContinuous} be satisfied. Then for each $0\le q\le p-1$, there is a Radon measure $\xi_q$ depending on $\kappa$ such that $\E{ \xi_{n,q } } \overset{v_c}{\to} \xi_q$ and $\xi_{n,q} \overset{v_c}{\to} \xi_q$ $a.s.$ as $n\to\infty$.
\end{corollary}

\section{Extensions to random fields}\label{Sec_ExtensionsToRandomFields}
We extend the theory from above to random fields in two settings, these correspond then to the situations discussed in Theorems~\ref{T:ConvergenceExpectationDiscrete} and \ref{T:ConvergenceExpectationContinuous} for the time series case.

The extension to random fields requires mainly notational changes. We consider stationary random fields indexed by the regular $d$-dimensional lattice $\Z^d$. The main difference is the ordering of the data which we assume to be located in the subset $\N^d$. If $u,v \in \N^d$ are two positions on the lattice, we write $u\ge v$ ($u \le v$) if and only if $u_i \ge v_i$ ($u_i \le v_i$) for all $i\in\{1,\ldots,d\}$. Moreover, we construct a total ordering on $\N^d$ with the $\ell^1$-norm as follows. Let $u,v\in\N^d$, then
\begin{align}\label{D:Ordering}
			u >_d v :\Longleftrightarrow \; \exists j \in \{1,\ldots,d\}: \|u_1^j\|_1 > \|v_1^j \|_1 \text{ and } \|u_1^k\|_1 = \|v_1^k\|_1 \text{ for } k = j+1,\ldots,d,
\end{align}
where $\|u_1^k\|_1 = |u_1|+\ldots+|u_k|$. The relations $>_d,\le_d,\ge_d$ follow in the same spirit.

For a vector $N=(N_1,\ldots,N_d)\in\N^d$, we denote the cardinality of the corresponding $d$-cube $\prod_{i=1}^d\{1,\ldots,N_i\}$ by $\pi(N)$. For a given a random field $X=(X_u:u\in\Z^d)$ and an $N\in\N^d$, we write $\mX_N$ for the associated point cloud $\{ X_u: u\le N\}$, which represents the sample data. In the following, we will consider only such $N\in\N^d$ which satisfy
\begin{align}\label{E:MinMaxQuot}
			\min\{N_i: i \in \{1,\ldots,d\} \} / \max\{ N_i: i \in\{1,\ldots,d\} \} \ge \bar{c},
\end{align}
for some constant $\bar{c}\in\R_+$. We write $N\to \infty$ for a sequence $N(k)\subseteq \N^d$ which satisfy the relation \eqref{E:MinMaxQuot} for each $N(k)$ and also fulfills $\max\{ N_i(k): i\in\{1,\ldots,d\} \}\to\infty$ as $k\to\infty$.

Consider a Marton coupling $( (X_u,X'_u): u\in\N^d )$ of a stationary random field on $\N^d$. We define the mixing matrix $\Gamma^{(\infty)} = (\Gamma^{(\infty)}_{u,v} )_{u,v\in\N^d}$ w.r.t.\ the ordering $>_d$. The line corresponding to location $u$ in the mixing matrix $\Gamma^{(\infty)}$ is given by
\begin{align}\label{C:MixingMatrixSpatial}
			\Gamma^{(\infty)}_{u,v} =	 \esssup_{ \text{ w.r.t.\ } \mu_u } \p\left( X_v^{(x_w: w <_d u, x_u,x'_u)} \neq {X'}_v^{(x_w: w <_d u, x_u,x'_u)} \right),
\end{align}
where $v \ge_d u$ and where $\mu_u$ is defined in the same spirit as $\mu_i$ in \eqref{Def:MuIMeasure}. 

We study the structure of the entries of the mixing matrix in a simple example. Consider a stationary random field $X$ on the lattice $\Z^2$ whose joint distribution can entirely be described by four (conditional) densities $f_{(0,0)}=\kappa, f_{(0,1)}, f_{(1,0)}$ and $f_{(1,1)}$. This means for any $N\in\N^d$ the joint distribution $\{X_u: u\le N\}$ can be simulated with these four (conditional) densities and we can do this also using the ordering $<_d$, beginning at the corner point (1,1). So we first simulate $X_{(1,1)}$ according to $\kappa = f_{(0,0)}$. All observations $X_{(1,t)}$ for $1< t \le N_1$ (resp. $X_{(t,1)}$ for $1< t \le N_2$) are simulated with $f_{(0,1)}$ (resp. $f_{(1,0)}$). All remaining observations are simulated with the conditional density $f_{(1,1)}$. Figure~\ref{fig:FactorizationScheme} illustrates the scheme. 

Consider a location $u$ in the lattice and a configuration of the Marton coupling which agrees at all locations of the past of $u$ w.r.t.\ $>_d$ (all locations $v$ with $u >_d v$). Consider a point $v$ in the future of $u$ w.r.t.\ $>_d$ (all locations $v$ with $v >_d u$). Then the distributions of $X_v^{(x_w: w <_d u, x_u,x'_u)}$ and ${X'}_v^{(x_w: w <_d u, x_u,x'_u)}$ are affected by the different configurations at location $u$ if and only if $v\ge u$. Hence, \eqref{C:MixingMatrixSpatial} is only affected by the locations $v$ which satisfy $v\ge u$, which is a strict subset of all those locations $v$ which satisfy $v \ge_d u$.

We come to the description of the dependence patterns. First we consider again the blocked density function from \eqref{E:BlockedDensity} and proceed as in the case for time series. Let $Z=(Z_u: u\in\Z^d)$ be a stationary random field on the regular $d$-dimensional lattice. The state space of $Z$ is discrete, i.e., $S=\{s_1,\ldots,s_{m^p}\}$, $s_i\neq s_j$, for $1\le i\neq j \le m^p$, such that $\p( Z_u = s_i ) = \alpha_i |A_i|$. Also $Z$ admits a Marton coupling whose mixing matrix $\Gamma^{(\infty)}$ from \eqref{C:MixingMatrixSpatial} satisfies (similar as in \eqref{C:MixingMatrix})
\begin{align}
		\|\Gamma^{(\infty)}\|_\infty <\infty.\label{C:MixingMatrix2} \tag{A7}
	\end{align}
\begin{figure}[h]
\centering
\begin{tikzpicture}[scale=1.5,line cap=round,line join=round,x=1cm,y=1cm]
\clip(.92,.92) rectangle (5.1,5.1);
\fill[line width=1pt,color=zzttqq,fill=zzttqq,fill opacity=0.0] (2.5,3.5) -- (3.5,3.5) -- (3.5,2.5) -- (5.5,2.5) -- (5.5,5.5) -- (2.5,5.5) -- cycle;
\draw [->,>=stealth,line width=1pt,color=qqqqcc] (1,1) -- (1,2);
\draw [->,>=stealth,line width=1pt,color=qqqqcc] (1,2) -- (1,3);
\draw [->,>=stealth,line width=1pt,color=qqqqcc] (1,3) -- (1,4);
\draw [->,>=stealth,>=stealth,line width=1pt,dash pattern=on 2pt off 4pt,color=qqqqcc] (1,4) -- (1,5);
\draw [->,>=stealth,line width=1pt,dash pattern=on 2pt off 4pt,color=qqqqcc] (1,5) -- (2,5);
\draw [->,>=stealth,line width=1pt,dash pattern=on 2pt off 4pt,color=qqqqcc] (1,4) -- (2,4);
\draw [->,>=stealth,line width=1pt,color=qqqqcc] (1,3) -- (2,3);
\draw [->,>=stealth,line width=1pt,color=qqqqcc] (1,2) -- (2,2);
\draw [->,>=stealth,line width=1pt,color=qqqqcc] (1,1) -- (2,1);
\draw [->,>=stealth,line width=1pt,color=qqqqcc] (2,1) -- (2,2);
\draw [->,>=stealth,line width=1pt,color=qqqqcc] (2,1) -- (3,1);
\draw [->,>=stealth,line width=1pt,color=qqqqcc] (3,1) -- (3,2);
\draw [->,>=stealth,line width=1pt,color=qqqqcc] (2,2) -- (3,2);
\draw [->,>=stealth,line width=1pt,color=qqqqcc] (2,2) -- (2,3);
\draw [->,>=stealth,line width=1pt,color=qqqqcc] (3,2) -- (3,3);
\draw [->,>=stealth,line width=1pt,color=qqqqcc] (2,3) -- (3,3);
\draw [->,>=stealth,line width=1pt,color=qqqqcc] (3,1) -- (4,1);
\draw [->,>=stealth,line width=1pt,color=qqqqcc] (4,1) -- (5,1);
\draw [->,>=stealth,line width=1pt,color=qqqqcc] (4,1) -- (4,2);
\draw [->,>=stealth,line width=1pt,color=qqqqcc] (3,2) -- (4,2);
\draw [->,>=stealth,line width=1pt,dash pattern=on 2pt off 4pt,color=qqqqcc] (4,2) -- (5,2);
\draw [->,>=stealth,line width=1pt,dash pattern=on 2pt off 4pt,color=qqqqcc] (5,1) -- (5,2);
\draw [->,>=stealth,line width=1pt,dash pattern=on 2pt off 4pt,color=qqqqcc] (5,2) -- (5,3);
\draw [->,>=stealth,line width=1pt,dash pattern=on 2pt off 4pt,color=qqqqcc] (4,2) -- (4,3);
\draw [->,>=stealth,line width=1pt,dash pattern=on 2pt off 4pt,color=ffqqqq] (3,3) -- (4,3);
\draw [->,>=stealth,line width=1pt,dash pattern=on 2pt off 4pt,color=ffqqqq] (4,3) -- (5,3);
\draw [->,>=stealth,line width=1pt,dash pattern=on 2pt off 4pt,color=ffqqqq] (5,3) -- (5,4);
\draw [->,>=stealth,line width=1pt,dash pattern=on 2pt off 4pt,color=ffqqqq] (4,3) -- (4,4);
\draw [->,>=stealth,line width=1pt,dash pattern=on 2pt off 4pt,color=ffqqqq] (3,3) -- (3,4);
\draw [->,>=stealth,line width=1pt,dash pattern=on 2pt off 4pt,color=qqqqcc] (2,3) -- (2,4);
\draw [->,>=stealth,line width=1pt,dash pattern=on 2pt off 4pt,color=qqqqcc] (2,4) -- (3,4);
\draw [->,>=stealth,line width=1pt,dash pattern=on 2pt off 4pt,color=ffqqqq] (3,4) -- (4,4);
\draw [->,>=stealth,line width=1pt,dash pattern=on 2pt off 4pt,color=ffqqqq] (4,4) -- (5,4);
\draw [->,>=stealth,line width=1pt,dash pattern=on 2pt off 4pt,color=ffqqqq] (5,4) -- (5,5);
\draw [->,>=stealth,line width=1pt,dash pattern=on 2pt off 4pt,color=qqqqcc] (2,4) -- (2,5);
\draw [->,>=stealth,line width=1pt,dash pattern=on 2pt off 4pt,color=qqqqcc] (2,5) -- (3,5);
\draw [->,>=stealth,line width=1pt,dash pattern=on 2pt off 4pt,color=ffqqqq] (3,4) -- (3,5);
\draw [->,>=stealth,line width=1pt,dash pattern=on 2pt off 4pt,color=ffqqqq] (3,5) -- (4,5);
\draw [->,>=stealth,line width=1pt,dash pattern=on 2pt off 4pt,color=ffqqqq] (4,4) -- (4,5);
\draw [->,>=stealth,line width=1pt,dash pattern=on 2pt off 4pt,color=ffqqqq] (4,5) -- (5,5);
\draw [line width=1pt,color=zzttqq] (2.5,3.5)-- (3.5,3.5);
\draw [line width=1pt,color=zzttqq] (3.5,3.5)-- (3.5,2.5);
\draw [line width=1pt,color=zzttqq] (3.5,2.5)-- (5.5,2.5);
\draw [line width=1pt,color=zzttqq] (2.5,5.5)-- (2.5,3.5);
\draw [fill=ffqqqq] (3,3) circle (1.5pt);
\end{tikzpicture}
\caption{Factorization scheme of the conditional distribution functions of a stationary Markov random field on $\Z^2$ which can be described with the conditional densities $f_{(0,0)}, f_{(1,0)}, f_{(0,1)}$ and $f_{(1,1)}$. The arrows indicate the direction of the evolution (simulation) of the field. Solid blue arrows point to these locations which lie in the past (w.r.t.\ $>_d$) of the red dot (and in the present, in case of the red dot itself). The dashed arrows (red and blue) point at locations which lie in the future (w.r.t.\ $>_d$) of the red dot. The location with the red dot marks a point $u$, where the Marton coupling $(X,X')$ disagrees conditional that it agrees on all locations from the past of $u$. The brown line separates the area, which is affected from the different states at location $u$, the corresponding arrows are red. Locations, which lie in the future, but are unaffected by the different states at location $u$, are marked in blue.}
\label{fig:FactorizationScheme}
\end{figure}
Define a new random field $X = (X_u: u\in\Z^d)$ with the help of $Z$ by
\begin{align}
			X_u = \sum_{i=1}^{m^p} Y_{u,i} \1{ Z_u = s_i}, \label{E:RFDiscrete} \tag{A8}
\end{align}
where the $Y_{u,i}$ are independent and uniformly distributed on $A_i$ for $u\in\Z^d$ and $1\le i\le m^d$. Then if $B\subseteq A_i$, $\p( X_u \in B ) = (|B| / |A_i|) \, (\alpha_i |A_i|) = \alpha_i |B|$. In particular, each $X_u$ has a density $\kappa$ on $[0,1]^p$. Also all other properties from the time series case are inherited. So, we have once more an invariance property as in \eqref{E:InvarianceOnCells}: For $N\in\N^d$ and $u,z_1,z_2\le N$ such that $z_2 <_d u <_d z_1$
\begin{align*}
			&\p\big(	(X_v: u\le_d v\le_d z_1, v\le N)\in A \,|\, X_w = x_w: z_2\le_d w <_d u, w\le N	\big) \\
			&= \p\big(	(X_v: u\le_d v\le_d z_1, v\le N)\in A \,|\, X_w = y_w: z_2\le_d w <_d u, w\le N	\big),
\end{align*}
for all $x_w,y_w\in A_{i_w}$, $1\le i_w\le m^p$, for all $w$ such that $z_2\le_d w <_d u$ and $w\le N	$.

Consequently, we obtain the following generalized variant of Theorem~\ref{T:ConvergenceExpectationDiscrete}.
\begin{theorem}\label{T:ConvergenceExpectationDiscreteRandomField}
Let $X=( X_u: u\in\Z^d )$ be a $[0,1]^p$-valued random field on $\Z^d$, which admits a Marton coupling that satisfies \eqref{C:MixingMatrix2} w.r.t.\ $>_d$ for each $N\in\N^d$. Each $X_u$ has a marginal density $\kappa$ as in \eqref{E:RFDiscrete} such that $0<\inf \kappa \le \sup \kappa < \infty$. Then for each $0\le q\le p-1$
\begin{align}
			\lim_{N\rightarrow \infty} \pi(N)^{-1} \E{ \beta^{r,s}_q( \cK( \pi(N)^{1/p} \mX_N) ) } &= \E{ \hat{b}_q(\kappa(X_u)^{1/p}(r,s)) }, \label{E:ConvergenceExpectationDiscreteRandomField0a}\\
					\lim_{N\rightarrow \infty}  \pi(N)^{-1} \beta^{r,s}_q( \cK( \pi(N)^{1/p} \mX_N) )  &= \E{ \hat{b}_q(\kappa(X_u)^{1/p}(r,s)) } \quad  a.s. \label{E:ConvergenceExpectationDiscreteRandomField0b}
			\end{align}
\end{theorem}

We refer to \cite{chazottes2007concentration, kulske2003concentration} who consider couplings for high-temperature Gibbs measures for the discrete random field $Z$ whose components take the values in $\{-,+\}$. Given certain upper bounds on the dependence within the random field, they obtain for the two state Gibbs model a coupling $(Z,Z')$ which satisfies
\begin{align}\label{E:ExpDecayDependenceMRF}
				\p( Z_v \neq Z'_v | Z_u = +, Z'_u = -, Z_w = Z'_w, \; \forall w <_d u ) \le e^{ - C \|x-y\|_1 },
\end{align}
for a certain constant $C\in\R_+$. So the probability of an unsuccessful coupling decays exponentially fast in the $\ell^1$-distance on the lattice, which is the minimal number of edges between $x$ and $y$ w.r.t.\ the standard $2^d$-neighborhood structure. In particular, the Marton coupling $(Z,Z')$ satisfies \eqref{C:MixingMatrix2}.

For a generalization of Theorem~\ref{T:ConvergenceExpectationContinuous} we need a decay assumption on the mixing matrices. In the case of a Markov chain of finite order, Theorem~16.0.2 in \cite{meyn2012markov} states that strictly positive and continuous conditional densities ensure uniform geometric ergodicity. So concerning the Marton coupling, we obtain a mixing matrix whose entries in one line decay at an exponential rate.

For random fields the situation is far more complex. To this end, we restrict ourselves to stationary Markov random fields $X$ of order 1 w.r.t.\ the $2^d$ neighborhood structure of the regular lattice $\Z^d$ whose joint distribution can be described with $2^d$ (conditional) density functions 
\begin{align}
			f_{s}\colon [0,1]^p\times [0,1]^{\|s\|_1 p} \rightarrow (0,\infty), \quad s\in\{0,1\}^d, \text{ where } \kappa = f_{0}\colon [0,1]^p\to (0,\infty). \label{E:ConditionalDensities} \tag{A9}
\end{align}
is the marginal density. More precisely, the distribution can be modeled with a scheme as in Figure~\ref{fig:FactorizationScheme}, however, on a $d$-dimensional lattice. The conditional density $f_s$ describes the transition within the set $\{z\in\Z^d: z_j=0 \text{ for } j\in J(s)\}$, where $J(s) = \{1\le j\le d: s_j=0\}$. 

We give an example for a cube $\cC_N =\{u\in\N ^d: u\le N\}$. First we can simulate the random variable $X_{(1,\ldots,1)}$ in the lower left corner according to $\kappa=f_0$. Let $e_k$ be the standard basis elements of $\R^d$ for $k=1,\ldots,d$, i.e., the vector whose $k$th entry is 1 and 0 otherwise. Then the conditional densities $f_{e_k}$ describe the transition on the coordinate axes of the cube. Similarly, with the remaining functions $f_s$, $s\neq (1,\ldots,1)$, we can completely simulate the transition on the lower envelope of the cube, i.e., the locations which are zero in at least one coordinate. Finally, the conditional density $f_{(1,\ldots,1)}$ describes the transition to those locations $u$, which are nonzero each entry. 

It is an important fact that due to the Markov structure we can factorize the distribution of the random field on $\cC_N$ with these conditional densities and use the ordering $>_d$ in the same time.

In contrast to the one-dimensional situation of a Markov chain, it is this time not enough that the conditional densities from \eqref{E:ConditionalDensities} are strictly positive in order to ensure a successful Marton coupling. To this end, we assume that the dependence within $X=\{X_u: u\in\Z^d\}$ decays at a polynomial rate in the sense that
\begin{align}
			\sum_{k=0}^\infty k^{\delta} \max_{v: \|v-u\|_{\max} = k } \Gamma^{(\infty)}_{u,v} < \infty, \label{E:DecayDependenceMRF} \tag{A10}
\end{align}
where $\delta > 3(d-1)$, $\|u\|_{\max} = \max\{ |u_i|: i=1,\ldots,d \}$ and where $\Gamma^{(\infty)}$ is the mixing matrix of the entire random field. Note that a uniform exponential decay as in \eqref{E:ExpDecayDependenceMRF} is obviously sufficient for \eqref{E:DecayDependenceMRF}. 
Note that due to the factorization property of $X$ from \eqref{E:ConditionalDensities}, the mixing matrix at position $(u,v)$ is nontrivial if and only if $v\ge u$. Also due to stationarity, it is entirely determined by the entries $\Gamma^{(\infty)}_{(1,\ldots,1),v}$, $v\ge (1,\ldots,1)$. Using last condition on the decay, we conclude with a generalized convergence result for persistent Betti numbers obtained from Markov random fields.

\begin{theorem}\label{T:ConvergenceExpectationMarkovRandomField}
Let the stationary random field $X$ be given by the (conditional) density functions $f_s$, $s\in\{0,1\}^d$, from \eqref{E:ConditionalDensities}, which are all continuous. Each $f_s$ is strictly positive on $[0,1]^p\times [0,1]^{\|s\|_1 p}$ in that $\inf_{x,y} f_s(x|y) > 0$. Let the mixing matrix of $X$ satisfy \eqref{E:DecayDependenceMRF}. Then $X$ fulfills the convergence results from \eqref{E:ConvergenceExpectationDiscreteRandomField0a} and \eqref{E:ConvergenceExpectationDiscreteRandomField0b}.
\end{theorem}

\section{Technical results}\label{Sec_TechnicalResults}

\subsection{Helpful tools}
Before we come to the proofs of the central results, we start with some auxiliary results.
\begin{lemma}[Concentration inequality for bounded transition kernels]\label{ConcentrationBoundedTransitionKernels}
Let $Z=\{Z_i:i\in\N\}$ be a sequence whose components $Z_i$ take values in the measure space $(S,\fS,\mu)$. Moreover assume that the conditional distributions $\cL(Z_i|Z_j=z_j, 1\le j<i)$ admit a conditional density $f_i$. These densities are uniformly bounded in the sense that the first part of \eqref{C:RegularityDensity} holds, i.e.,
$$
			\sup_{i\in\N} \| f_i(\cdot | \cdot ) \|_{\mu \otimes \mu^{\otimes (i-1)}, \infty}\le f^* < \infty.
$$
Let $(B_n:n\in\N) \subseteq S $ be a sequence of measurable sets such that $\limsup_{n\rightarrow \infty }n \mu(B_n) (\log n)^{-\alpha} \le c^*$, for certain $\alpha,c^*\in\R_+$. Then there is a constant $c_1 \in\R_+$ such that for all $n\in\N$ and $t\in\R_+$
\begin{align}\begin{split}\label{Eq:ConcentrationBoundedTransitionKernels}
		\p\Big(	\sum_{i=1}^{n} \1{Z_i \in B_n } > t	\Big) &\le e^{-t} \ \E{ e^{	\sum_{i=1}^{n} \1{Z_i \in B_n }} } \\
		&\le \exp(-t  + (e-1) f^* n \mu(B_n) ) \\
		&\le c_1 \, \exp(-t + c^*(e-1)f^* (\log n)^\alpha ).
\end{split}\end{align}
In particular, let $Z$ be an $\R^p$-valued homogenous Markov chain which admits uniformly bounded conditional densities. For each $n$, let $B_n = B(x,r_n)$ be the $r_n$-neighborhood of a point $x\in\R^p$ w.r.t.\ the Euclidean distance such that $n r_n^p \rightarrow c^*$. Then \eqref{Eq:ConcentrationBoundedTransitionKernels} holds with $\alpha = 0$.
\end{lemma}
\begin{proof}
First, we bound the Laplace transform of $ \1{Z_i \in B_n}$ w.r.t.\ $\cF_{i-1}$, where $\{\cF_i \}_i$ is the natural filtration of the process $Z$ with $\cF_0$ being trivial. We have
\begin{align*}
		\E{ \exp\left(	\1{Z_i \in B_n}	\right) | \cF_{i-1} } &\le e  \, \p(Z_i \in B_n| Z_1,\ldots,Z_{i-1} ) + \left[ 1 - \p(Z_i \in B_n| Z_1,\ldots,Z_{i-1} )  \right] \\
		&\le 1 + (e-1) f^* \mu(B_n) \le \exp\left( (e-1) f^* \mu(B_n)  \right).
\end{align*}
Thus, we obtain for the entire process $	\p(	\sum_{i=1}^n \1{Z_i \in B_n}  > t	) \le \exp(-t  + (e-1) f^* n \mu(B_n) )$, using Markov's inequality. This finishes the proof because $\limsup_{n\rightarrow \infty }n \mu(B_n) (\log n)^{-\alpha} \le c^*$.
\end{proof}

The next lemma is a generalization of Lemma~3.1 in \cite{yogeshwaran2017random}.
\begin{lemma}\label{L:SimplicesAndMeasure}
Let $j\in\N$ and $X=\{ X_t: t\in\Z \}$ be a process which takes values in a measure space $(S,\fS,\mu)$ with a non-atomic measure $\mu$. Let $\{v_1,\ldots,v_\ell\}$ be a set of $\ell\le j$ distinct natural numbers. Assume that the distribution of the vector $(X_{v_1},\ldots,X_{v_{\ell}})$, when conditioned on another observation $X_i$, $i\notin \{v_1,\ldots,v_l\}$, admits a density $f_{(X_{v_1},\ldots,X_{v_{\ell}}) | X_i }$ w.r.t.\ $\mu^{\otimes\ell}$. Assume that these densities are uniformly essentially bounded in the sense that the second part of \eqref{C:RegularityDensity} holds, i.e.,
\[
		 \sup_{i\in \Z} \sup_{ \substack{ (v_1,\ldots,v_{\ell} ) \in \Z^{\ell} \\ i\neq v_{j} \neq v_k } } \| f_{(X_{v_1},\ldots,X_{v_{\ell}}) | X_i } \|_{ \mu^{\otimes(\ell+1)}, \infty } \le f^* < \infty, \quad \forall\, \ell \le j.
\]
Then for each $A\in\fS$ and for each $r>0$ it is true that
\begin{align}\label{E:SimplicesAndMeasure0}
			\E{ K_j( (\eta_n \mX_n )\cap A, r ) } \le \E{	K_j ( \eta_n \mX_n, r; A ) } &\le (f^*)^2 	n^{j+1} \, \mu(\eta_n^{-1} A) \, \sup_{x\in S} \mu(B(x,2 \eta_n^{-1} r))^{j} \\
		\text{ and }	\quad \E{ K_j( \eta_n \mX_n , r; \eta_n X_t ) } &\le f^* \	n^{j} \, \sup_{x\in S} \mu(B(x,2 \eta_n^{-1} r))^{j}.\label{E:SimplicesAndMeasure0b}
\end{align}
In particular, if $S$ is a subset of $\R^p$, $\eta_n = n^{1/p}$ and $\mu$ equals the Lebesgue measure, then 
\eqref{E:SimplicesAndMeasure0} is of constant order $O( |A| r^{p j} )$ and \eqref{E:SimplicesAndMeasure0} is of order $O( r^{p j} )$ .
\end{lemma}	
\begin{proof}
We only prove the statement in \eqref{E:SimplicesAndMeasure0}, the statement in \eqref{E:SimplicesAndMeasure0b} follows in the same fashion. The first inequality in \eqref{E:SimplicesAndMeasure0} is obvious. Thus, we only show the second one. Observe that
\begin{align}\label{E:SimplicesAndMeasure1}
		K_j ( \eta_n \mX_n, r; A ) \le \sum_{i=1}^n \1{ \eta_n X_i \in A } \sum_{ \substack{ (u_1,\ldots,u_{j}) \in \{1,\ldots, n\}^{j}:\\ i \neq u_\ell \neq u_{\ell'} } } \; \prod_{\ell=1}^{j} \1{ d(X_i,X_{u_\ell} ) \le 2 \eta_n^{-1} r }
\end{align}
because the distance between any two points in a $j$-simplex in the {\v C}ech or the Vietoris-Rips complex $K_j ( \eta_n \mX_n, r; A )$ is at most $2r$. On the one hand, 
\begin{align}\label{E:SimplicesAndMeasure2}
		\E{	\prod_{\ell=1}^{j} \1{d(X_i,X_{u_\ell}) \le 2 \eta_n^{-1} r } \Big| X_i	} \le f^* \ \sup_{x\in S} \mu(B(x,2 \eta_n^{-1} r))^{j},
\end{align}
and also
$
	\# \{	(u_1,\ldots,u_{j}) \in \{1,\ldots, n\}^{j}: i \neq u_\ell \neq u_{\ell'} \neq i, \: 1\le \ell, \ell' \le j 	\} \le n^{j}.
	$
On the other hand
\begin{align}\label{E:SimplicesAndMeasure3}
		\E{	\1{ \eta_n X_i \in A } } \le f^* \mu\big(\eta_n^{-1}	A	\big).
 \end{align}
Combining \eqref{E:SimplicesAndMeasure2}, \eqref{E:SimplicesAndMeasure3} with \eqref{E:SimplicesAndMeasure1} yields the conclusion.
\end{proof}

The following result is well-known to topologists.
\begin{lemma}[Geometric Lemma, Lemma 2.11 in \cite{hiraoka2018limit}] \label{L:GeometricLemma}
Let $\mX \subseteq \mY$ be two finite point sets in $S$. Then
$
		|\beta^{r,s}_q (\cK(\mY)) - \beta^{r,s}_q (\cK(\mX))	| \le \sum_{j=q}^{q+1}  |K_j(\mY,s) \setminus K_j(\mX,s)|.
$
\end{lemma}

\subsection{Technical details on Section~\ref{Sec_MainResults}}
We come to the proof of Theorem~\ref{T:ExponentialInequality}. Similar as in \cite{yogeshwaran2017random}, we use a result of \cite{chalker1999size} to establish an exponential inequality without the need of bounding the martingale differences in the supremum-norm. 

\begin{proof}[Proof of Theorem~\ref{T:AbstractExpIneq}]
Consider the natural filtration of the process $X$, $\cF_i = \sigma(X_{1},\ldots,X_i)$ for $i=0,\ldots,n$ with the convention that $\cF_0 = \{\emptyset, \Omega\}$. We rewrite $H_n(\mX_n)-\E{H_n(\mX_n)}$ in terms of martingale differences as follows
\begin{align}\begin{split}\label{Eq:AbstractExpIneq2}
		H_n(\mX_n)-\E{H_n(\mX_n)} &= \sum_{i=1}^n V_{n,i},
\end{split}\end{align}
where
$V_{n,i} =\E{ H_n(\mX_n) | \cF_i } -  \E{H_n(\mX_n) | \cF_{i-1} }$.
An abstract result of \cite{chalker1999size} yields 
\begin{align}\begin{split}\label{Eq:AbstractExpIneq3}
				\p\Big( \Big| \sum_{i=1}^n V_{n,i} \Big| > b_1		\Big) &\le 2 \exp\left(	- \frac{b_1^2}{32 n b_2^2}	\right)  \\
				&\quad + \left(	1+ \frac{2 \sup_{1 \le i \le n} \norm{V_{n,i}}_{\p,\infty} }{ b_1 } \right) \sum_{i=1}^n \p\left( |V_{n,i}| > b_2		\right)
\end{split}\end{align}
for any $b_1,b_2 \in\R_+$. Hence, it remains to compute bounds of $V_{n,i}$. 
In all cases, we have the universal bound from \eqref{E:UniversalBound}. So, $\norm{V_{n,i}}_{\p,\infty}\le c_1 n^q$.

Next, we investigate the probabilities on the right-hand side of \eqref{Eq:AbstractExpIneq3}. Define for $a\in S $ and $i\in\{1,\ldots,n\}$
\begin{align*}
		I_{n,i}(a) & \coloneqq \int_{ {S}^{n-i}} \p\left( X_{i+1}^n \in \intd{x}_{i+1}^n \,|\, X_{1},\ldots, X_{i-1},  X_{i}= a 	\right) \, H_n( \{X_1^{i-1}, a, x_{i+1}^n \} ).
\end{align*}
Write $\nu_i$ for the conditional distribution of $ X_i$ given $(X_{1},\ldots,X_{i-1})$ on $S$, viz.,
\[
		\nu_i = \mathbbm{M}_{X_i|(X_{1},\ldots,X_{i-1})} \left( (X_{1},\ldots,X_{i-1}), \cdot \right).
\]
Then, it follows with elementary calculations that for each $1\le i\le n$
\[
		V_{n,i} \le \esssup_{ \text{ w.r.t.\ } \nu_i } I_{n,i}(\cdot) - \essinf_{ \text{ w.r.t.\ } \nu_i} I_{n,i}(\cdot) \quad a.s.
\]
Let $\epsilon>0$ be arbitrary but fixed. Choose $a^*,b^*\in S$ such that $I_{n,i}(a^*) \ge \esssup_{ \text{ w.r.t.\ } \nu_i } I_{n,i}(\cdot) - \epsilon / 2$ and $I_{n,i}(b^*) \le \essinf_{\text{ w.r.t.\ } \nu_i  } I_{n,i}(\cdot) + \epsilon/2$. Consider the Marton coupling of $(X_1,\ldots,X_n)$ and write $\mX_n^{(X_1,\ldots,X_{i-1},a^*,b^*)}$ for the point cloud associated to the coupling element $X^{(X_1,\ldots,X_{i-1},a^*,b^*)}$. The notation ${\mX'}_n^{(X_1,\ldots,X_{i-1},a^*,b^*)}$ is used for the point cloud obtained from the counterpart $X'^{(X_1,\ldots,X_{i-1},a^*,b^*)}$. Consequently,
\begin{align}
		V_{n,i}-\epsilon &\le I_{n,i}(a^*) - I_{n,i}(b^*) \nonumber\\
		&= \E{ H_n( \mX_n^{(X_1,\ldots,X_{i-1},a^*,b^*)} )  - H_n( {\mX'}_n^{(X_1,\ldots,X_{i-1},a^*,b^*)} )} \nonumber \\
		\begin{split}\label{Eq:AbstractExpIneq4}
		&= \int_{{S}^{n-i} \times {S}^{n-i} } \p_{ \left( X^{(X_{1},\ldots,X_{i-1},a^*,b^* )}, {X'}^{(X_{1},\ldots,X_{i-1},a^*,b^* )} \right) }\left( \intd{ ( y_{i+1}^n, {y'}_{i+1}^n)}		\right) \\
		&\qquad\qquad \times \left\{	H_n( \{X_1^{i-1},a^*, y_{i+1}^n\} ) - H_n( \{X_1^{i-1},b^*, y_{i+1}^n\} ) \right\},
\end{split}
\end{align}
(by abusing the notation slightly). We write $\mY_{n,i} =  \{X_1^{i-1},a^*, y_{i+1}^n\}$ and $\mY'_{n,i} =  \{X_1^{i-1},b^*, {y'}_{i+1}^n\}$ and consider the difference of the functionals in  \eqref{Eq:AbstractExpIneq4}. The point clouds $\mY_{n,i}$ and ${\mY'}_{n,i}$ in \eqref{Eq:AbstractExpIneq4} differ at most in $n-i+1$ entries for each $i$. These entries are $\{a^*, y_{i+1},\ldots,y_n\}$ and $\{b^*, y'_{i+1},\ldots,y'_n\}$. Thus, we can transform $\mY_{n,i}$ into ${\mY'}_{n,i}$ in $n-i+1$ steps exchanging one entry in each step, i.e., we consider the transformations
\begin{align}\label{Eq:AbstractExpIneq5b}
			 \mY_{n,i}^{(1)}=\mY_{n,i} \leftrightsquigarrow  \mY_{n,i}^{(2)} \leftrightsquigarrow \ldots \leftrightsquigarrow \mY_{n,i}^{(n-i+2)}={\mY'}_{n,i},
\end{align}
where $ \mY_{n,i}^{(\ell)} = \{X_1^{i-1}, b^*, {y'}_{i+1}^{i+\ell-2}, y_{i+\ell-1}^n \}$, for $\ell\in\{2,\ldots,n-i+2\}$.

Using this definition, the difference of the functionals in \eqref{Eq:AbstractExpIneq4} is bounded above by
\begin{align}\label{Eq:AbstractExpIneq5c}
		|H_n( \mY_{n,i} ) - H_n( {\mY'}_{n,i} ) | \le \sum_{\ell = 1}^{n-i+1} \left| H_n( \mY_{n,i}^{(\ell+1) } ) - H_n( \mY_{n,i}^{(\ell)} ) \right|.
\end{align}
The symmetric difference $\mY_{n,i}^{(\ell+1)} \triangle \mY_{n,i}^{(\ell)} $ is at most $\{a^*,b^*\}$ for $\ell=1$ and $\{y_{\ell+i-1},y'_{\ell+i-1} \}$ for $\ell \in \{2,\ldots, n-i+1\}$. Let $\ell \in\{2,\ldots, n-i+1\}$; clearly, if $y_{\ell+i-1} = y'_{\ell+i-1}$, then $| H_n(\mY_{n,i}^{(\ell+1) }  ) - H_n( \mY_{n,i}^{(\ell)}	) |=0$.
If $y_{\ell+i-1} \neq y'_{\ell+i-1}$, we can use \eqref{E:ExchOneCost}, which states that the exchange-one cost are local, to obtain 
\begin{align}\begin{split}\label{Eq:AbstractExpIneq6}
		 \big|  H_n ( \mY_{n,i}^{(\ell+1) } ) - H_n ( \mY_{n,i}^{(\ell)}  ) \big| &\le  c_2 \Big| \big(	\mY_{n,i}^{(\ell+1) } \cup \mY_{n,i}^{(\ell) }\big) \cap B(y_{\ell+i-1}, \eta_n^{-1} r ) \Big|^{\wt q} \\
		&\quad +  c_2 \Big|	\big( \mY_{n,i}^{(\ell+1) } \cup \mY_{n,i}^{(\ell) } \big)\cap B(y'_{\ell+i-1}, \eta_n^{-1} r ) \Big|^{\wt q},
		\end{split}
\end{align}
A similar argument applies to $| H_n ( \mY_{n,i}^{(2) } ) -  H_n (\mY_{n,i}^{(1)}) |$, which admits the same bound as in \eqref{Eq:AbstractExpIneq6} using the points $a^*,b^*$ 

Write $\sN(r)=\sN(r,S,d)$ for the $r$-covering number of $S$ w.r.t.\ $d$ from Condition~\ref{C:CoveringCondition}. Use the family of coverings $\{ \{ B(w_j,r): 1\le j\le \sN(r) \}: r>0\}$ to define for each $r>0$ and each $u>0$ the set
\begin{align}\label{Eq:AbstractExpIneq5}
		\mA_{n,u}(r) \coloneqq \{	x\in S^n: \#[ B(w_j,2r)\cap x] \le u,\: \forall \, 1\le j\le \sN(r) \}.
\end{align}
Loosely speaking, when considering only sets in $\mA_{n,u}(r)$, we can control the degree of accumulation within a point cloud on $S$ if we choose  $u<<n$ for some fixed $r$.  Using the definition of the set $\mA_{n,u}(r)$, \eqref{Eq:AbstractExpIneq4} is at most
\begin{align*}
		& \int_{{S}^{n-i} \times {S}^{n-i} } \p_{ \left( X^{(X_{1},\ldots,X_{i-1},a^*,b^* )}, {X'}^{(X_{1},\ldots,X_{i-1},a^*,b^* )} \right) }\left( \intd{ ( y_{i+1}^n, {y'}_{i+1}^n)}		\right) \Biggl\{ \\
		&\quad \Big( \1{ \mY_{n,i} \notin \mA_{n,u}( \eta_n^{-1} r)  } + \1{ {\mY'}_{n,i} \notin \mA_{n,u}(\eta_n^{-1} r)  } \Big) c_1 n^q   \\
		&\quad +  \1{ \mY_{n,i} \in \mA_{n,u}(\eta_n^{-1} r), \mY'_{n,i} \in \mA_{n,u}(\eta_n^{-1} r)  } \Biggl[ \\
		&\quad \Biggl( \sum_{\ell=2}^{n-i+1}  \1{y_{\ell+i-1} \neq y'_{\ell+i-1} } \Biggl[ c_2 \Big| \big(	\mY_{n,i}^{(\ell+1) } \cup \mY_{n,i}^{(\ell) }\big) \cap B(y_{\ell+i-1}, \eta_n^{-1} r ) \Big|^{\wt q} \\
		&\quad +  c_2 \Big|	\big( \mY_{n,i}^{(\ell+1) } \cup \mY_{n,i}^{(\ell) } \big)\cap B(y'_{\ell+i-1}, \eta_n^{-1} r ) \Big|^{\wt q} \Biggl] \Biggl) \\
		&\quad +  \Biggl( c_2 \Big| \big(	\mY_{n,i}^{(2) } \cup \mY_{n,i}^{(1) }\big) \cap B(a^*, \eta_n^{-1} r ) \Big|^{\wt q}  +  c_2 \Big|	\big( \mY_{n,i}^{(2) } \cup \mY_{n,i}^{(1) } \big)\cap B(b^*, \eta_n^{-1} r ) \Big|^{\wt q} \Biggl) \Biggl] \Biggl\}. 
\end{align*}
Clearly, if both $\mY_{n,i}$ and ${\mY'}_{n,i}$ are in $\mA_{n,u}(\eta_n^{-1} r)$, then each point cloud of the type $\mY_{n,i}^{(\ell)}$ and $\mY_{n,i}^{(\ell+1)}\cup \mY_{n,i}^{(\ell)}$ is in $\mA_{n,2u}(\eta_n^{-1} r)$.

By Condition~\ref{C:CoveringCondition} there is a covering $\{ B(w_k, \eta_n^{-1}r): 1\le j\le \sN(\eta_n^{-1} r ) \}$ of $S$. Consequently, there is a $1\le k\le \sN(\eta_n^{-1}r)$ such that $y_{\ell+i-1} \in B(w_k,\eta_n^{-1}r)$ and the neighborhood $B(y_{\ell+i-1},\eta_n^{-1}r)$ is contained in $ B(w_k,2\eta_n^{-1}r) $.
So, in the case where $\mY_{n,i}$ and ${\mY'}_{n,i}$ are both in $\mA_{n,u}(\eta_n^{-1} r)$, we have
\[
	\Big| \big(	\mY_{n,i}^{(\ell+1) } \cup \mY_{n,i}^{(\ell) }\big) \cap B(y_{\ell+i-1}, \eta_n^{-1} r ) \Big|^{\wt q} \le \Big| \big(	\mY_{n,i}^{(\ell+1) } \cup \mY_{n,i}^{(\ell) }\big) \cap B(w_k, 2\eta_n^{-1} r ) \Big|^{\wt q} \le (2u)^{\wt q},
\]
the same applies to the neighborhoods of $y'_{\ell+i-1}, a^*, b^*$. Thus, we obtain an upper bound of the following type for the integral in \eqref{Eq:AbstractExpIneq4}
\begin{align}
				&\int_{S^{n-i} \times S^{n-i}}  \p_{ \left( [X^{(X_{1},\ldots,X_{i-1},a^*,b^* )}]_{i+1}^n, [{X'}^{(X_{1},\ldots,X_{i-1},a^*,b^* )}]_{i+1}^n \right) }\left( \intd{ ( y_{i+1}^n, {y'}_{i+1}^n)}	\right) \nonumber \\
		&\quad \left\{	c_1 n^q( \1{ \mY_{n,i} \notin \mA_{n,u}(\eta_n^{-1} r)  } + \1{ {\mY'}_{n,i} \notin \mA_{n,u}(\eta_n^{-1} r)  } ) + 2c_2 (2u)^{\wt q} \sum_{j=i}^n \1{ y_j \neq y'_j  } \right\} \nonumber \\
		\begin{split}\label{Eq:AbstractExpIneq7}
		&= c_1 n^q \Big( \p(  \mX_n \notin \mA_{n,u}(\eta_n^{-1} r) | X_1^{i-1}, X_i = a^*) + \p( \mX_n \notin \mA_{n,u}(\eta_n^{-1} r) | X_1^{i-1}, X_i = b^*) \Big)  \\
		&\qquad\qquad\qquad + 2 c_2 (2u)^{\wt q} \sum_{j=i}^n \p\left(	 X_j^{(X_{1},\ldots,X_{i-1},a^*,b^* )} \neq  {X'}_j^{(X_{1},\ldots,X_{i-1},a^*,b^* )}	\right).
		\end{split}
\end{align}
The last term in \eqref{Eq:AbstractExpIneq7} is at most $2c_2 (2u)^{\wt q} \gamma_\infty$ uniformly in $n,i$ and $a^*,b^*$ due to the condition from \eqref{C:MixingMatrix}, which implies
\[
		\sup_{n\in\N} \sup_{i\in\{1,\ldots,n\}} \sum_{j=1}^n \p(	 X_j^{(X_{1},\ldots,X_{i-1},a^*,b^* )} \neq  {X'}_j^{(X_{1},\ldots,X_{i-1},a^*,b^* )}	) \le \gamma_\infty \quad a.s.
\]
Moreover as $\epsilon$ was arbitrary, this last bound from \eqref{Eq:AbstractExpIneq7} is also true for the limit and we have
\begin{align}
		|V_{n,i}| &\le \esssup_{ \text{ w.r.t.\ } \nu_i } I_{n,i}(\cdot) - \essinf_{\text{ w.r.t.\ } \nu_i  } I_{n,i}(\cdot) \nonumber \\
		\begin{split}\label{Eq:AbstractExpIneq7b}
		&\le c_1 n^q \Big( \p(  \mX_n \notin \mA_{n,u}(\eta_n^{-1} r) | X_1^{i-1}, X_i = a^*) \\
		&\quad  + \p( \mX_n \notin \mA_{n,u}(\eta_n^{-1} r) | X_1^{i-1}, X_i = b^*) \Big) + 2 c_2 \gamma_\infty (2u)^{\wt q}	
		\end{split}
\end{align}
where all constants are uniform in $n$ and $i$. We fix the value of $u$ in the following as $u \coloneqq n^\gamma$ and return to \eqref{Eq:AbstractExpIneq4}. We choose $b_2 = 4 c_2 \gamma_\infty (2u)^{\wt q}$. Thus, using \eqref{Eq:AbstractExpIneq7b}
\begin{align}
			\p\left( |V_{n,i}| > b_2		\right) &\le  \p\left( c_1 n^q \p\left(\mX_{n} \notin \mA_{n,u}(\eta_n^{-1}r) \,|\, X_{1}^{i-1}, X_i = a^* \right) >  b_2 / 4	\right) \nonumber \\
			&\quad\qquad\qquad +  \p\left( c_1 n^q \p\left( \mX_n \notin \mA_{n,u}(\eta_n^{-1} r) \,|\, X_1^{i-1}, X_i = b^* \right) > b_2/ 4 \right) \nonumber \\
			\begin{split}\label{Eq:AbstractExpIneq8}
			& \le  \frac{4 c_1 n^q}{ b_2 } \Biggl[ \E{ \p\left( \mX_{n} \notin \mA_{n,u}(\eta_n^{-1}r) \,|\, X_{1}^{i-1}, X_i = a^* \right) } \\
					&\quad\qquad\qquad + \E{\p\left( \mX_{n} \notin \mA_{n,u}(\eta_n^{-1}r) \,|\, X_{1}^{i-1}, X_i = b^* \right)}\Biggl],
\end{split}
\end{align}
where we use Markov's inequality in the last step. We bound above both expectations in \eqref{Eq:AbstractExpIneq8} as follows: First note that for $r\ge 0$ and $w\in S$
\begin{align*}
	\p( e^{\1{X_k\in B(w,r)} } \,|\, X_{1}^{i-1}, X_i = a, X_{i+1}^{k-1} ) \le \exp( f^*(e-1) \mu( B(w,r)) ) \quad a.s.
\end{align*}
In particular, we have for each $r>0$
\[
	\p( e^{\sum_{k>i} \1{X_k\in B(w,r)} } \,|\, X_{1}^{i-1}, X_i = a ) \le \exp( (n-i) f^*(e-1) \mu( B(w,r)) ).
\]
Additionally, we apply Lemma~\ref{ConcentrationBoundedTransitionKernels} to $\E{ \exp( \sum_{k< i} \1{X_k\in B(w,r)} )}$ and use that $\1{X_i\in B(w,r)}$ is at most 1. Then we obtain for a state $a \in S$
\begin{align*}
		&\E{ \p\left( \mX_{n} \notin \mA_{n,u}(\eta_n^{-1}r) \,|\, X_{1}^{i-1}, X_i = a \right) } \\
		&=\E{ \p\left(	 X_{1}^{n} \cap B(w_j,2\eta_n^{-1} r ) > u \text{ for one } j \le \sN(\eta_n^{-1}r) \,|\, X_{1}^{i-1}, X_i = a \right) }  \\
		&\le \sN(\eta_n^{-1}r) \ \max_{1\le j\le \sN(\eta_n^{-1}r) } \E{ \p\left( \sum_{k=1}^n \1{X_k\in B(w_j,2\eta_n^{-1}r)} > u \,\Big|\, X_{1}^{i-1}, X_i = a \right)	}  \\
		&\le \sN(\eta_n^{-1}r) \  \exp\big( (n-1) f^* (e-1) \ \sup_{w\in S} \mu( B(w,2\eta_n^{-1}r)) - u + 1	\big).
\end{align*}
Combining this last inequality with \eqref{Eq:AbstractExpIneq8}, we see that
\begin{align*}
	&\p\left( |V_{n,i}| > b_2		\right) \\
	& \le \frac{8 c_1 n^q \ \sN(\eta_n^{-1}r) }{4 c_2 \gamma_\infty (2 n^\gamma )^{\ol q}} \  \exp\big( (n-1) f^* (e-1) \ \sup_{w\in S} \mu( B(w,2\eta_n^{-1}r)) - n^\gamma + 1	\big).
\end{align*}
Moreover, inserting this result in \eqref{Eq:AbstractExpIneq3} for the above choice of $b_2$ and $b_1= n^a t$, yields
\begin{align*}
				&\p\big( \big| \sum_{i=1}^n V_{n,i} \big| > n^a t	\big)  \\
				&\le 2 \exp\left(	-  \frac{n^{2a} t^2}{ 32 n (4c_2\gamma_\infty (2 n^\gamma )^{ \wt q})^2 } \right) + \left(	1+  \frac{2c_1 n^q}{n^a t} \right) \, \frac{8 c_1 n^{q+1} \ \sN(\eta_n^{-1}r) }{4 c_2 \gamma_\infty (2 n^\gamma )^{\wt q}} \\
				&\quad   \times \exp\big\{ (n-1) f^* (e-1) \ \sup_{w\in S} \mu( B(w,2\eta_n^{-1}r)) - n^\gamma + 1	\big\} \\
				&\le 2 \exp\left(	-  \frac{n^{2(a-\gamma \wt q)-1} t^2 }{ 4^{\wt q}2 (16 c_2 \gamma_\infty )^2} \right) + \frac{2 c_1 e n^{2q+1 -\gamma \wt q-a}}{c_2 2^{\wt q}\gamma_\infty} \Big( \frac{1}{n^{q-a}} + \frac{2c_1}{t} \Big) \\
				&\quad  \times \exp\Big\{ - \big[ n^\gamma - \log\sN(\eta_n^{-1}r) - f^*(e-1) \ n \sup_{w\in S} \mu( B(w,2\eta_n^{-1}r)) 		\big]		\Big\}.
\end{align*}
Finally, applying the definition of $\gamma$ completes the proof.
\end{proof}

\begin{proof}[Proof of Theorem~\ref{T:ExponentialInequality}]
It remains to verify that the persistent Betti function $\beta^{r,s}_q$ satisfies the condition in \eqref{E:UniversalBound} and in \eqref{E:ExchOneCost}. It follows from the definition of Betti numbers that \eqref{E:UniversalBound} is satisfied for $c_1 = 1$ and  the exponent $q$.

Next, we inspect the condition in \eqref{E:ExchOneCost}. Let $\mY, \mZ$ be two point clouds of $n$ points, which differ in exactly one point, viz., $\mY \Delta \mZ = \{y,z\}$. We can use the Geometric Lemma (Lemma~\ref{L:GeometricLemma}) to obtain 
\begin{align}
		\begin{split}\label{Eq:ExponentialInequalityBettiNumbers6}
		\Big| \beta_q^{r,s} \Big( \cK ( \eta_n\mY ) \Big) - \beta_q^{r,s} \Big( \cK ( \eta\mZ  )	\Big) \Big|  &\le \sum_{j=q,q+1} K_j\Big(	\mY \cup \mZ , \eta_n^{-1} s; \{y\}	\Big) \\
		&\quad + \sum_{j=q,q+1} K_j\Big( \mY\cup\mZ	, \eta_n^{-1}s; \{z\}	\Big),
		\end{split}
\end{align}
where we use for the last inequality the scaling relation $K_j( \eta (\mY\cup\mZ) ,r) = K_j(\mY\cup\mZ,\eta^{-1}r)$, which is valid for the {\v C}ech and the Vietoris-Rips complex for all $\eta>0$ because of the homogeneity of $d$.

Observe that a $j$-simplex in the filtration $K(	\mY\cup\mZ, \eta_n^{-1}s )$ has a diameter of at most $2\eta_n^{-1}s$. Thus, a $j$-simplex with a node in a point $y$, resp.\ $z$, lies in the $(2\eta_n^{-1}s)$-neighborhood of $y$, resp.\ $z$. Consequently, \eqref{Eq:ExponentialInequalityBettiNumbers6} is at most
\[
	2 | (\mY\cup \mZ) \cap B(y,2\eta_n^{-1}s) |^{q+1} + 2 | (\mY\cup \mZ) \cap B(z,2\eta_n^{-1}s) |^{q+1}.
\]
Hence, the condition in \eqref{E:ExchOneCost} is satisfied with $c_2 = 2$, $r = 2s$ and exponent $\wt q = q+1$.
\end{proof}

\begin{proof}[Proof of Theorem~\ref{T:ConvergenceExpectationDiscrete}]
We split the proof in three parts. We show in the first part that
\[
			n^{-1} \E{ \beta_q^{r,s}( \cK( n^{1/d} \mX_n)) } = \sum_{i=1}^{m^p} n^{-1} \E{ \beta_q^{r,s}( \cK( n^{1/d} (\mX_n \cap A_i) ) ) }  + o(1), \quad n\to\infty.
\]
Define a filtration, which is the union of the single filtrations when restricted to the cubes $A_i$, by the complexes
\[
			\mathring{K}(n^{1/p} \mX_n,r) = \bigcup_{i=1}^{m^p} K(n^{1/p} (\mX_n \cap A_i) , r ), \quad r\ge 0.
\]
Since this union is of disjoint complexes, we have $\beta_q^{r,s} ( \mathring{\cK}(n^{1/p} \mX_n) ) = \sum_{i=1}^{m^p} \beta_q^{r,s} ( \cK (n^{1/p} (\mX_n \cap A_i) ) )$.
We use Lemma~\ref{L:GeometricLemma} and Lemma~\ref{L:SimplicesAndMeasure} to arrive at 
\begin{align*}
				&n^{-1} \E{	|	\beta_q^{r,s} ( \cK (n^{1/p} \mX_n ) )  - \beta_q^{r,s} ( \mathring{\cK}(n^{1/p} \mX_n)  )  | } \\
			&\le	 n^{-1} \sum_{j=q,q+1} \E{	K_j\Big( (n^{1/p}\mX_n) \cap \Big(\bigcup_{i=1}^{m^p} (\partial (n^{1/p} A_i) )^{(2s )} \Big) , s \Big)		} \\
			&\le 2 (f^*)^2 |B(0,2s)|^{q+2} m^p \frac{(n^{1/p}+4s)^p-n}{n} =: R_{n}.
\end{align*}
Then $R_{n}$ is of order $n^{-1/p}$. So, we can consider the expectation on the blocks $A_i$ instead. 

From now let $i\in\{1,\ldots,m^p\}$ be an arbitrary but fixed index. Write $\ell_{i,1},\ldots,\ell_{i,p}$ for the edge lengths of $A_i$. So that $|A_i|$ equals $\prod_{j=1}^p \ell_{i,j}$. Also write $M_i$ for the diagonal matrix $diag(\ell_{i,j}: 1\le j \le p)$. Note that $\operatorname{det}(M_i) = |A_i|$. This completes the first part.

In the second part, we use McDiarmid's inequality from Theorem~\ref{Thrm:McDiarmidIneq}. Set $S_{n,i} = \sum_{t=1}^n \1{X_t\in A_i}$ and $h(n) = n^{3/4}$. Since $(X_1,\ldots,X_n)$ admits a Marton coupling which satisfies \eqref{C:MixingMatrix}, we can apply Theorem~\ref{Thrm:McDiarmidIneq} to arrive at
\[
			\p( | S_{n,i} - \E{S_{n,i}}| \ge h(n)) \le 2 \exp\Big( - \frac{2n^{1/2} }{\gamma_\infty } \Big).
\]
Using the definition $I_{n,i} = [-h(n) + \E{S_{n,i}}, h(n) + \E{S_{n,i}}]$ and the fact that the Betti numbers of dimension $q$ are polynomially bounded by $n^{q+1}$, we obtain 
\begin{align*}
						&\Big| n^{-1} \E{ \beta_q^{r,s}( \cK( n^{1/p} (\mX_n \cap A_i) ) ) } - \sum_{k\in I_{n,i}} n^{-1} \E{ \1{S_{n,i}=k} \beta_q^{r,s}( \cK( n^{1/p} (\mX_n \cap A_i) ) ) } \Big| \\
						&\le n^{-1} n^{q+1} \p( |S_{n,i}- \E{S_{n,i}}| \ge h(n)) \le 2 n^q \ \exp\Big( - \frac{2n^{1/2} }{\gamma_\infty } \Big). 
\end{align*}
In particular, $n^{-1}\mathbb{E}[ \beta_q^{r,s}( \cK( n^{1/p} (\mX_n \cap A_i) ) ) \{ S_{n,i} \notin I_{n,i} \} ]$ is negligible and we can focus in the following on the restriction $ \{ S_{n,i} \in I_{n,i} \}$. For this purpose, write $\mu_{n,i} = \floor{ \E{S_{n,i}} } = \floor{ n \alpha_i |A_i| }$, then it follows from Lemma 5.5 in \cite{krebs2018asymptotic} that for each $0\le r \le s$
\begin{align}\begin{split}\label{E:UnifConvergenceBettiPoisson}
		&\limsup_{n\rightarrow \infty} \sup_{k\in I_{n,i} } n^{-1} \Big| \E{ \beta_q^{r,s}( \cK( n^{1/p} ( X'_1,\ldots,X'_k ) ) )} \\
		&\qquad\qquad\qquad\qquad - \E{ \beta_q^{r,s}( \cK( n^{1/p} ( X'_1,\ldots,X'_{\mu_{n,i}}) ) ) } \Big| = 0,
\end{split}\end{align}
where the $X'_1, X'_2,\ldots$ are independent and uniformly distributed on $[0,1]^p$. We will use \eqref{E:UnifConvergenceBettiPoisson} later.

In the third part, we study the success runs of $(X_t: 1\le t\le n)$ and the sum $S_{n,i}$: If an $X_t$ falls in $A_i$, we term this a success and a failure otherwise. Consider a path with exactly $k$ successes $J = (\mathscr{F}_1, \mathscr{S}_1, \mathscr{F}_2, \ldots, \mathscr{F}_v, \mathscr{S}_v, \mathscr{F}_{v+1}) \in \{0,1\}^n$, where $v\le k$ and where each $\mathscr{S}_i$ is a sequence  of 1's and each $\mathscr{F}_i$ a sequence of 0's (potentially $\mathscr{F}_1$ and $\mathscr{F}_{v+1}$ have length 0) for $i \in\{1,\ldots v\}, (\text{resp. in } \{1,\ldots v+1\})$. So, on the path $J$, we have $S_{n,i}=k$.

Consider the expectation on this path $J$. For this write $J^*$ for the index set which contains the positions in $J$ that mark a success. Write $\mM_{t}$ for the conditional distribution of $X_t$ given the past $X_1,\ldots,X_{t-1}$. Then
\begin{align}
			&n^{-1} \E{ \1{X_1^n\in J} \beta_q^{r,s}( \cK( n^{1/p} (\mX_n \cap A_i) ) ) } \nonumber \\
			\begin{split}\label{E:ConvergenceExpectationDiscrete1}
			&=n^{-1} \int_{[0,1]^p} \p_{X_1}(\intd{x_1}) \int_{[0,1]^p} \mM_{2}(\intd{x_2}|x_1) \ldots \int_{[0,1]^p} \mM_{n}(\intd{x_n}|x_1,\ldots,x_{n-1} )  \\
			&\quad\quad\quad\quad \times \beta_q^{r,s}( \cK( n^{1/p} (x_1^n \cap A_i) ) ) \prod_{i\in J^*} \1{x_\ell \in A_i} \prod_{i\in J\setminus J^*} \1{x_\ell \notin A_i}.
\end{split}\end{align}			
Consider the situation for the last success which is given at a position $t^*$. Note that each $\mM_{t}$ admits a conditional density $f_t$ because the distribution of $X_t$ on each block $A_j$, $j\in\{1,\ldots,m^p\}$, is uniform and independent of the past observations $X_1,\ldots,X_{t-1}$ given that $X_t$ falls in the block $A_j$. So this $f_t(x_t|x_1,\ldots,x_{t-1})$ is constant for all $x_t$ from a block $A_j$. Due to the blocked structure of the conditional densities of $X$ and the invariance property from \eqref{E:InvarianceOnCells}, the contribution of the observations $X_{t^*},\ldots, X_n$ to the integral in \eqref{E:ConvergenceExpectationDiscrete1} is then
\begin{align*}
			&\int_{[0,1]^p} \mM_{t^*}(\intd{x}_{t^*}|x_1,\ldots, x_{t^*-1}) \, \beta^{r,s}_q(  \cK( n^{1/p} (x_{J^*\setminus \{t^*\}},x_{t^*})  ) ) \,\\
			&\qquad\qquad\qquad \p( X_{t^*+1}\notin A_i,\ldots,X_n \notin A_i \,|\, x_1,\ldots, x_{t^*} ) \\
			&= f_{t^*}(z_{t^*}|x_1,\ldots,x_{t^*-1})  \, \p( X_{t^*+1}\notin A_i,\ldots,X_n \notin A_i \,|\, x_1,\ldots, x_{t^*-1}, z_{t^*} )\; \\
			&\qquad\qquad\qquad \int_{A_i}  \intd{x}_{t^*} \beta^{r,s}_q(  \cK( n^{1/p} (x_{J^*\setminus \{t^*\}},x_{t^*})  ) ) \\
		 &=\p( X_{t^*}\in A_i, X_{t^*+1}\notin A_i,\ldots,X_n \notin A_i \,|\, x_1,\ldots, x_{t^*-1} ) \\
		 &\qquad\qquad\qquad \int_{A_i} \intd{x}_{t^*} |A_i|^{-1} \beta^{r,s}_q(  \cK( n^{1/p} (x_{J^*\setminus \{t^*\} }, x_{t^*})  ) ),
\end{align*}	
where $z_{t^*}$ is an arbitrary but fixed element of $A_i$ and where we use for the last equation that 
\[
	f_{t^*} (z_{t^*}|x_1,\ldots,x_{t^*-1}) \ |A_i| = \p(X_{t^*}\in A_i \,|\, x_1,\ldots, x_{t^*-1}).
	\]
Using recursively this conditional independence argument, one obtains for \eqref{E:ConvergenceExpectationDiscrete1}
\begin{align}
			&n^{-1} \E{ \1{X_1^n\in J} \beta_q^{r,s}( \cK( n^{1/p} (\mX_n \cap A_i) ) ) } \nonumber \\
			&= n^{-1} \p( X_1^n \in J) \int_{A_i} \ldots \int_{A_i} \intd{ x_1 } \ldots \intd{ x_k } \,  |A_i|^{-k} \beta_q^{r,s}( n^{1/p}( x_1,\ldots, x_k)) \nonumber \\
			&=\p( X_1^n \in J) \ n^{-1} \, \E{ \beta_q^{r,s}( n^{1/p}  \{ M_i X'_1,\ldots, M_i X'_k) \} }, \label{E:ConvergenceExpectationDiscrete2}
\end{align}
where the $X'_t$ are independent and uniformly distributed on $[0,1]^p$. 

Moreover, using the uniform approximation result from \eqref{E:UnifConvergenceBettiPoisson} shows that \eqref{E:ConvergenceExpectationDiscrete2} equals
\begin{align}\label{E:ConvergenceExpectationDiscrete2b}
			&\p( X_1^n \in J) \ \Big( n^{-1} \  \E{ \beta_q^{r,s}( n^{1/p}  \{ M_i X'_1,\ldots, M_i X'_{\mu_{n,i}}\} )  } + o(1) \Big), \quad n\to\infty,
\end{align}
where the remainder $o(1)$ is uniform in $J$ and $k\in I_{n,i}$. Furthermore, using the dilatation rules of the expectation of persistent Betti numbers computed from the {\v C}ech or Vietoris-Rips filtration, we obtain for the main term of \eqref{E:ConvergenceExpectationDiscrete2b}
\begin{align*}
			& \p( X_1^n \in J) \; (\alpha_i |A_i|) \; \frac{\alpha_i^{-1} |A_i|^{-1}}{n} \\
			&\qquad\qquad\qquad\E{ \beta_q^{\alpha_i^{1/p}(r,s)} \big( (n\alpha_i |A_i| )^{1/p}  \{ (M_i |A_i|^{-1/p}) X'_1,\ldots, (M_i |A_i|^{-1/p}) X'_{\mu_{n,i}} \} \big) } \\
			&= \p( X_1^n \in J) \Big( ( \alpha_i |A_i| )\; \hat{b}_q( \alpha_i^{1/p}(r,s) ) + o(1) \Big), \quad n\to \infty,
\end{align*}
where the remainder $o(1)$ is uniform in $J$ and where the last equality follows as in the proof of Lemma 10 in \cite{divol2018persistence}. Summing over all paths $J$ with exactly $k$ successes, over all $k\in I_{n,i}$ and over all $i=1,\ldots,m^p$ yields then the conclusion, viz.,
\[
			\lim_{n\to\infty} n^{-1} \E{ \beta_q^{r,s}( \cK( n^{1/p} \mX_n)) } = \sum_{i=1}^{m^p} ( \alpha_i |A_i| )\;  \hat{b}_q( \alpha_i^{1/p}(r,s) )  = \E{ \hat{b}_q( \kappa(X_t)^{1/p}(r,s) ) }.
\]
This proves the first assertion in \eqref{E:LimitDiscrete}. Combining this last statement with Theorem~\ref{T:ExponentialInequality} and the Borel-Cantelli-Lemma shows the second assertion in \eqref{E:LimitDiscreteB}. So, the proof is complete.
\end{proof}

\begin{proof}[Proof of Theorem~\ref{T:ConvergenceExpectationContinuous}]
In the proof, we sometimes abuse the notation slightly in order to keep formulas shorter. To be precise, we write $K(U_1^n,r)$ for the simplicial complex $K( \{U_1,\ldots,U_n \},r)$ of a vector $U_1,\ldots,U_n$ at filtration time $r$ to save space. The related expressions are abbreviated in this way, too.

In the first step of the proof, we construct a discrete Markov chain of order $m$, $\wt{X}$, which approximates $X$ closely. To this end, let $\epsilon>0$ be arbitrary but fixed. We use a discrete density function $g_\epsilon$, which is an approximation of the joint density $g$ of $(X_1,\ldots,X_{m+1})$. We write $f_{t}$ for the conditional density of $X_t$ given $(X_{t-1},\ldots, X_1)$ with the convention that $f_1$ is the marginal density $\kappa$. 

Since we assume that the process $X$ is a Markov chain of order $m$, we are actually dealing with the conditional densities $f_1 \equiv \kappa, f_2,\ldots, f_{m+1}$ only. Using the approximation $g_\epsilon$, we obtain approximations $f_{\epsilon,1} =: \kappa_\epsilon, f_{\epsilon,2},\ldots,f_{\epsilon,m+1}$, which are defined in the same spirit as the $f_t$. We choose the precision between $g$ and $g_\epsilon$ sufficiently high (in the $\|\cdot\|_\infty$-norm) such that
\begin{align}\begin{split}\label{EQ:ConvergenceExpectationContinuous1} 
				&\big| f_{\epsilon,t}(x_i| x^{i-1}_{i-t+1}) - f_{t}(x_i| x^{i-1}_{i-t+1}) \big| \le \epsilon, \\
				&\quad \forall x=(x_{i-t+1},\ldots,x_{i-1},x_i)\in [0,1]^{tp}, \quad \forall \, t=1,\ldots, m+1.
\end{split}\end{align}
Thus, at each step of the evolution of the Markov chain, we can approximate each conditional density with a discrete conditional density at a precision of at least $\epsilon/2$ (measured in the total variation distance). Note that this is possible because we assume that $\inf\{ g(z): z\in [0,1]^{(m+1)p}\} >0$, so that all conditional densities are well defined.

We write $\wt{X}$ for the Markov chain of order $m$ obtained from the above $\epsilon$-approximation scheme, note that we can also choose $g_\epsilon$ to be strictly positive. In particular, this implies that $\wt{X}$ satisfies the assumptions of Theorem~\ref{T:ConvergenceExpectationDiscrete} because it is uniformly geometrically ergodic, see \cite{meyn2012markov} Theorem~16.0.2. Clearly, also $X$ is uniformly geometrically ergodic, hence, $X$ admits a Marton coupling (see also Example~\ref{MarkovChainDelayEmbedding} and \cite{paulin2015concentration} Proposition 2.4), whose mixing matrix $\Gamma^{(n)}$ (based on $X_1,\ldots,X_n$) satisfies
\begin{align}\label{EQ:ConvergenceExpectationContinuous2}
			\sup_{n\in\N} \max_{1\le i \le n} \sum_{j=i}^n (j-i)^\delta \ \Gamma^{(n)}_{i,j} < \infty, \quad \forall \delta \ge 0.
			\end{align}
In the second step, we use the decomposition
\begin{align}
		& \left|n^{-1} \E{ \beta^{r,s}_q( \cK( n^{1/p} \mX_n) ) } - \E{ \hat{b}_q(\kappa(Y)^{1/p}(r,s)) } \right| \nonumber \\
		&\le n^{-1} \left| \E{ \beta^{r,s}_q( \cK( n^{1/p} \mX_n) ) } - \E{ \beta^{r,s}_q( \cK( n^{1/p} \wt\mX_n) ) } \right| \label{EQ:ConvergenceExpectationContinuous3} \\
		&\quad +  \left|n^{-1}  \E{ \beta^{r,s}_q( \cK( n^{1/p} \wt\mX_n) ) } - \E{ \hat{b}_q(\kappa_\epsilon(Y_\epsilon)^{1/p}(r,s)) } \right| \label{EQ:ConvergenceExpectationContinuous4} \\
		& \quad + \left| \E{ \hat{b}_q(\kappa\epsilon(Y_\epsilon )^{1/p}(r,s)) } - \E{ \hat{b}_q(\kappa(Y)^{1/p}(r,s)) } \right|, \label{EQ:ConvergenceExpectationContinuous5}
\end{align}
where the random variables $Y_\epsilon$ (resp. $Y$) have density $\kappa_\epsilon$ (resp. $\kappa$). 

If $\epsilon$ converges to 0, \eqref{EQ:ConvergenceExpectationContinuous5} converges to 0, too, see as well \eqref{E:LimitContinuity} and \cite{divol2018persistence} for details. Moreover, from Theorem~\ref{T:ConvergenceExpectationDiscrete}, we conclude that \eqref{EQ:ConvergenceExpectationContinuous4} converges to 0 as $n$ tends to $\infty$ for each $\epsilon>0$ and corresponding approximation $g_\epsilon$.

So, the term in \eqref{EQ:ConvergenceExpectationContinuous3} remains. Let $w\in \N$ and $\ol w$ be such that $1/w + 1/ \ol w =1$. For the remainder of this proof, we show that \eqref{EQ:ConvergenceExpectationContinuous3} is at most $C^*_{s,p,w} \epsilon^{1/\ol w} + o(1)$ uniformly in $n$, where the constant $C^*_{s,p, w}$ does neither depend on the choice of the approximation parameter $\epsilon$ nor on $n$. It solely depends on $s,p$ and $w$. For this we rewrite the expectations in \eqref{EQ:ConvergenceExpectationContinuous3} as
\begin{align}
\begin{split}\label{EQ:ConvergenceExpectationContinuous6}
			&\E{ \beta^{r,s}_q( \cK( n^{1/p} \mX_n) ) }  \\
			&= \int \intd{x_1} f_1(x_1)\,  \int \intd{x_2} f_2(x_2\,|\, x_1)\, \ldots \int \intd{x_m} f_m(x_m\,|\, x^{m-1}_1)\, \\
			&\quad \int \intd{x_{m+1} } f_{m+1}(x_{m+1} \,|\, x^{m}_1) \ldots \int \intd{x_{n} } f_{m+1}(x_{n} \,|\, x^{n-1}_{n-m}) \beta^{r,s}_q( \cK( n^{1/p} x_1^n) ), \end{split} \\
			&\nonumber \\
			\begin{split}\label{EQ:ConvergenceExpectationContinuous7}	
			&\E{ \beta^{r,s}_q( \cK( n^{1/p} \wt\mX_n) ) }  \\
			&= \int \intd{x_1} f_{\epsilon,1}(x_1)\,  \int \intd{x_2} f_{\epsilon,2}(x_2\,|\, x_1)\, \ldots \int \intd{x_m} f_{\epsilon,m} (x_m\,|\, x^{m-1}_1)\, \\
			&\quad \int \intd{x_{m+1} } f_{\epsilon,m+1}(x_{m+1} \,|\, x^{m}_1) \ldots \int \intd{x_{n} } f_{\epsilon,m+1}(x_{n} \,|\, x^{n-1}_{n-m})  \beta^{r,s}_q( \cK( n^{1/p} x_1^n) ). \end{split}
\end{align}
We transform \eqref{EQ:ConvergenceExpectationContinuous6} in \eqref{EQ:ConvergenceExpectationContinuous7} in $n$-steps using a specific coupling in each step. For this purpose we write the difference between \eqref{EQ:ConvergenceExpectationContinuous6} and \eqref{EQ:ConvergenceExpectationContinuous7} as a telescopic sum as follows (the exchanged factor is given in square parentheses)
\begin{align}
			& 	\E{ \beta^{r,s}_q( \cK( n^{1/p} \mX_n) ) -  \beta^{r,s}_q( \cK( n^{1/p} \wt\mX_n) )} \nonumber \\
			&= \int \intd{x_1} \big[f_1(x_1) - f_{\epsilon,1}(x_1) \big]\,  \int \intd{x_2} f_2(x_2\,|\, x_1)\, \ldots \int \intd{x_m} f_m(x_m\,|\, x^{m-1}_1)\, \nonumber\\
			&\quad \qquad\qquad \int \intd{x_{m+1} } f_{m+1}(x_{m+1} \,|\, x^{m}_1) \ldots \int \intd{x_{n} } f_{m+1}(x_{n} \,|\, x^{n-1}_{n-m}) \; \beta^{r,s}_q( \cK( n^{1/p} x_1^n) ) \nonumber \\
			&\quad +  \int \intd{x_1} f_{\epsilon,1}(x_1)\,  \int \intd{x_2} \big[f_2(x_2\,|\, x_1) - f_{\epsilon,2}(x_2\,|\, x_1) \big]\, \ldots \int \intd{x_m} f_m(x_m\,|\, x^{m-1}_1)\, \nonumber \\
			&\quad  \qquad\qquad \int \intd{x_{m+1} } f_{m+1}(x_{m+1} \,|\, x^{m}_1) \ldots \int \intd{x_{n} } f_{m+1}(x_{n} \,|\, x^{n-1}_{n-m}) \; \beta^{r,s}_q( \cK( n^{1/p} x_1^n) ) \nonumber \\
			&\quad + \ldots \nonumber \\
			\begin{split}\label{EQ:ConvergenceExpectationContinuous8}
			&\quad + \int \intd{x_1} f_{\epsilon,1}(x_1)\,  \int \intd{x_2}  f_{\epsilon,2}(x_2\,|\, x_1) \, \ldots \int \intd{x_m} f_{\epsilon,m} (x_m\,|\, x^{m-1}_1)\, \\
			&\quad \qquad  \int \intd{x_{m+1} } f_{\epsilon, m+1}(x_{m+1} \,|\, x^{m}_1) \ldots \int \intd{x_{t} } \big[ f_{m+1}(x_{t} \,|\, x^{t-1}_{t-m}) -  f_{\epsilon, m+1}(x_{t} \,|\, x^{t-1}_{t-m}) \big] \ldots  \\
			&\quad \qquad \ldots  \int \intd{x_{n} } f_{m+1}(x_{n} \,|\, x^{n-1}_{n-m}) \; \beta^{r,s}_q( \cK( n^{1/p} x_1^n) ) 
			\end{split}\\
				&\quad + \ldots \nonumber\\
			&\quad + \int \intd{x_1} f_{\epsilon,1}(x_1)\,  \int \intd{x_2}  f_{\epsilon,2}(x_2\,|\, x_1) \, \ldots \int \intd{x_m} f_{\epsilon,m}(x_m\,|\, x^{m-1}_1)\, \nonumber\\
			&\quad \qquad\int \intd{x_{m+1} } f_{\epsilon, m+1}(x_{m+1} \,|\, x^{m}_1) \nonumber\\
			&\quad \qquad \ldots \int \intd{x_{n} } \big[ f_{m+1}(x_{n} \,|\, x^{n-1}_{n-m}) - f_{\epsilon,m+1}(x_{n} \,|\, x^{n-1}_{n-m}) \big] \;\beta^{r,s}_q( \cK( n^{1/p} x_1^n) ). \nonumber
\end{align}
Each integral in the sum can be interpreted as a difference between the expectation of two persistent Betti numbers of two coupled processes $(Z'_{t,\cdot}, \wt Z_{t,\cdot})$ for $t\in\{1,\ldots,n\}$. We explain this coupling in three steps and refer to the term in \eqref{EQ:ConvergenceExpectationContinuous8}, which shows the general situation. First, the $t$th coupling starts with $\wt{Z}_{t,1}={Z'}_{t,1},\ldots,\wt{Z}_{t,t-1}={Z'}_{t,t-1}$; so $Z'_{t,\cdot}$ and $\wt Z_{t,\cdot}$ have the same distribution as the stationary discrete Markov chain $\wt{X}$ (with the densities $f_{\epsilon,\cdot})$ from time 1 to $t-1$. 

Second, at time $t$, we simulate a random variable $Z'_{t,t}$ using the conditional density $f_i$ (where the index $i$ depends on the position of $t$). Also at time $t$, we simulate a random variable $\wt{Z}_{t,t}$ using the conditional density $f_{\epsilon,i}$. Note that $Z'_{t,t}$ and $\wt{Z}_{t,t}$ can be coupled such that
\begin{align}\label{EQ:ConvergenceExpectationContinuous9}
	2 \p( Z'_{t,t} \neq \wt{Z}_{t,t} | \wt{Z}_{t,1},\ldots,\wt{Z}_{t,t-1} ) \le \frac{1}{2} \int_{[0,1]^p } |f_{i}(y| \wt{Z}_{t,1}^{t,t-1}) - f_{\epsilon,i}(y| \wt{Z}_{t,1}^{t,t-1}) | \intd{y} \le \frac{\epsilon}{2} \quad a.s.
\end{align}
because of the choices in \eqref{EQ:ConvergenceExpectationContinuous1}; we refer to \cite{den2012probability} for an abstract maximal coupling result on Polish spaces.

Third, we find two chains $Z'_{t,j}$ and $\wt{Z}_{t,j}$, $j=t+1,\ldots$, using the conditional densities $f_i$ such that the single elements at time $j\le n$ satisfy $\p( Z'_{t,j} \neq \wt{Z}_{t,j} | \wt{Z}_{t,1},\ldots, \wt{Z}_{t,t-1}, \wt{Z}_{t,t}, Z'_{t,t}) \le \Gamma^{(n)}_{t,j}$. This last inequality follows from the properties of the Marton coupling, see \eqref{EQ:ConvergenceExpectationContinuous2}.

In the following, we will use the abbreviation $\wt Z^{t,j}_{t,i}$ for the vector $(\wt Z_{t,i},\ldots,\wt Z_{t,j})$; we use the notation in the same spirit for $Z'_{t,\cdot}$.
Using the above coupling, we see that \eqref{EQ:ConvergenceExpectationContinuous3} is at most
\begin{align}
		& n^{-1} \sum_{t=1}^{n} \E{ \big| \beta^{r,s}_q( \cK( n^{1/p} {Z'}_{t,1}^{t,n} ) ) - \beta^{r,s}_q( \cK( n^{1/p} {\wt Z}_{t,1}^{t,n} ) ) \big| } \label{EQ:ConvergenceExpectationContinuous10}
\end{align}
So for each $t$ the point clouds differ at most in the points ${Z'}_{t,t}^{t,n}$ resp. ${\wt Z}_{t,t}^{t,n}$ and we can always transform one point cloud in the other in $n-t+1$ steps.

Regarding \eqref{EQ:ConvergenceExpectationContinuous10}, we show that there is a $C\in\R_+$ such that for each $t\in \{1,\ldots,n\}$ and for each $n\in\N$ the coupling $(Z'_{t,\cdot}, \wt{Z}_{t,\cdot})$ satisfies
\begin{align}\label{EQ:ConvergenceExpectationContinuous13}
			\E{ \big| \beta^{r,s}_q( \cK( n^{1/p} {Z'}_{t,1}^{t,n} ) ) - \beta^{r,s}_q( \cK( n^{1/p} \wt{Z}_{t,1}^{t,n} ) ) \big|} \le C \epsilon^{1/\ol w}.
\end{align}
Define the coupling time between $({Z'}_{t,j})_j$ and $(\wt{Z}_{t,j})_j$ by
\begin{align*}
			\tau_c(t) = \inf\{j\ge t:  Z'_{t,j} = \wt{Z}_{t,j}, Z'_{t,j-1} = \wt{Z}_{t,j-1} \ldots, Z'_{t,j-m+1} = \wt{Z}_{t,j-m+1}	\},
\end{align*}
i.e., for all $j\ge \tau_c(t)$ the chains evolve again in lockstep, viz., $Z'_{t,j} = \wt{Z}_{t,j}$ for $j\ge \tau_c(t)$. Note that given $Z'_{t,t} \neq \wt{Z}_{t,t}$, we have $\tau_c(t) \ge t+m$.

The coupling times $\tau_c$ admit a tail bound which involves the coefficients from the Marton coupling as follows: If $u\ge t+m$, then
\begin{align}\begin{split}\label{EQ:ConvergenceExpectationContinuous14}
			\p( \tau_c(t) > u \,|\, Z'_{t,t}\neq \wt{Z}_{t,t} ) &\le \p( Z'_{t,u}\neq \wt{Z}_{t,u} \text{ or }  Z'_{t,u-1}\neq \wt{Z}_{t,u-1} \\
			&\quad\qquad \text{ or } \ldots \text{ or } Z'_{t,u-m+1} \neq \wt{Z}_{t,u-m+1} \,|\, Z'_{t,t}\neq \wt{Z}_{t,t}  ) \\
			&\le \p( Z'_{t,u}\neq \wt{Z}_{t,u} \,|\, Z'_{t,t}\neq \wt{Z}_{t,t}  )  + \p(  Z'_{t,u-1}\neq \wt{Z}_{t,u-1} \,|\, Z'_{t,t}\neq \wt{Z}_{t,t}  )   \\
			&\quad\qquad + \ldots + \p( Z'_{t,u-m+1} \neq \wt{Z}_{t,u-m+1} \,|\, Z'_{t,t}\neq \wt{Z}_{t,t}  ) \\
			&\le \Gamma^{(n)}_{t,u} + \Gamma^{(n)}_{t,u-1} + \ldots + \Gamma^{(n)}_{t,u-m+1}.
\end{split}\end{align}
Since the coefficients of the Marton coupling satisfy \eqref{EQ:ConvergenceExpectationContinuous2}, this shows that the moments of the coupling times (when conditioned on $\{ Z'_{t,t}\neq \wt{Z}_{t,t} \}$) are uniformly bounded: We have for $\delta\ge 0$
\begin{align}
	&\int_t^\infty (u-t)^\delta \ \p( \tau_c(t) > u \ | \{ Z'_{t,t}\neq \wt{Z}_{t,t}  \} ) \intd{u} \nonumber\\
	&\le \sum_{u=t}^\infty (u+1-t)^\delta \ \p( \tau_c(t) > u \ | \{ Z'_{t,t}\neq \wt{Z}_{t,t}  \} ) \nonumber\\ 
	&\le m^{1+\delta} + \sum_{u=m}^\infty (u+1-t)^\delta \ \big( \Gamma^{(n)}_{t,u} + \ldots + \Gamma^{(n)}_{t,u-m+1} \big).\label{EQ:ConvergenceExpectationContinuous15}
\end{align}
We begin our considerations with the restriction to the event $\{ \tau_c(t)\le n \}$ for $t\in\{1,\ldots,n\}$:
\begin{align}
				&\E{ \1{\tau_c(t) \le n } \Big| \beta^{r,s}_q( \cK( n^{1/p} {Z'}_{t,1}^{t,n} ) ) - \beta^{r,s}_q( \cK( n^{1/p} \wt{Z}_{t,1}^{t,n} ) ) \Big| } \nonumber \\
				\begin{split}\label{EQ:ConvergenceExpectationContinuous16}
				&= \sum_{j=t+m}^{n} \mathbb{E}\Big[ \1{\tau_c(t) = j } \1{ {Z'}_{t,t} \neq \wt{Z}_{t,t} } \Big| \beta^{r,s}_q( \cK( n^{1/p} \{ \wt{Z}_{t,1}^{t,t-1}, {Z'}_{t,t}, {Z'}_{t,t+1}^{t,j-1}, {Z'}_{t,j}^{t,n} \}) ) \\
				&\qquad\qquad\qquad\qquad\qquad\qquad - \beta^{r,s}_q( \cK( n^{1/p} \{ \wt{Z}_{t,1}^{t,t-1}, \wt{Z}_{t,t}, \wt{Z}_{t,t+1}^{t,j-1}, {Z'}_{t,j}^{t,n} \}) ) \Big| \Big], \end{split}
\end{align}
where we use for the last equality that $\tau_c(t)$ can be at least $t+m$ conditional on the event $\{ {Z'}_{t,t} \neq \wt{Z}_{t,t}  \}$. We study the expectations in \eqref{EQ:ConvergenceExpectationContinuous16}. First we apply the Geometric Lemma to obtain
\begin{align}
				&\mathbb{E}\Big[ \Big| \beta^{r,s}_q( \cK( n^{1/p} \{ \wt{Z}_{t,1}^{t,t-1}, {Z'}_{t,t}, {Z'}_{t,t+1}^{t,j-1}, {Z'}_{t,j}^{t,n} \}) ) \nonumber \\
				&\quad - \beta^{r,s}_q( \cK( n^{1/p} \{ \wt{Z}_{t,1}^{t,t-1}, \wt{Z}_{t,t}, \wt{Z}_{t,t+1}^{t,j-1}, {Z'}_{t,j}^{t,n} \}) ) \Big| \1{ \tau_c(t) = j , {Z'}_{t,t} \neq \wt{Z}_{t,t} }   \Big] \nonumber \\
				\begin{split}\label{EQ:ConvergenceExpectationContinuous17}
				&\le \sum_{\ell=1}^{j-t} \sum_{k=q}^{q+1} \E{ K_k\Big( \wt{Z}_{t,1}^{t,t-1+\ell} \cup {Z'}_{t,t-1+\ell}^n, n^{-1/p} s; {Z'}_{t,t+\ell-1}	\Big)  \1{ \tau_c(t) = j , {Z'}_{t,t} \neq \wt{Z}_{t,t} }  } \\
			&\quad + 	\sum_{\ell=1}^{j-t} \sum_{k=q}^{q+1} \E{ K_k\Big( \wt{Z}_{t,1}^{t,t-1+\ell} \cup {Z'}_{t,t-1+\ell}^n, n^{-1/p} s; \wt{Z}_{t,t+\ell-1}	\Big)  \1{ \tau_c(t) = j , {Z'}_{t,t} \neq \wt{Z}_{t,t} }  }.
			\end{split}
\end{align}
Consider the $w$th moment of a simplex count $K_k$ for some $\ell$. For this purpose denote by $(Y_0,\ldots, Y_n)$ the data $(\wt{Z}_{t,1},\ldots,\wt Z_{t,t-1+\ell}, {Z'}_{t,t-1+\ell},\ldots,Z'_n)$. Then one finds with elementary combinatorial arguments that the $w$th moment of the simplex count in the first line in \eqref{EQ:ConvergenceExpectationContinuous17} is at most
\begin{align}
	&\E{ K_k\Big( \wt{Z}_{t,1}^{t,t-1+\ell} \cup {Z'}_{t,t-1+\ell}^n, n^{-1/p} s; {Z'}_{t,t+\ell-1}	\Big)^w } \nonumber \\
	&= \mathbb{E}\Biggl[ \sum_{ \substack{ u^{(1)} \in \{0,\ldots, n\}^{k}:\\ t-1+\ell \neq u^{(1)}_\ell \neq u^{(1)}_{\ell'} } } \; \prod_{\ell=1}^{k} \1{ d( {Z'}_{t,t-1+\ell} ,Y_{u^{(1)}_\ell} ) \le 2 n^{-1/p} s }  \nonumber \\
	&\qquad\qquad\qquad \ldots \sum_{ \substack{ u^{(w)} \in \{0,\ldots, n\}^{k}:\\ t-1+\ell \neq u^{(w)}_\ell \neq u^{(w)}_{\ell'} } } \; \prod_{\ell=1}^{k} \1{ d(Z'_{t,t-1+\ell},Y_{u^{(w)}_\ell} ) \le 2 n^{-1/p} s } \Biggl] \nonumber \\
	\begin{split}\label{EQ:ConvergenceExpectationContinuous18}
	&\le \sum_{v=k}^{wk} \binom{n}{v} \sum_{\substack{u_1,\ldots,u_v \ge 1 \\ \sum_{u_i = w k}} } \frac{(wk)!}{ u_1! \ldots u_v!} \ \mathbb{E} \big[ \1{d(Z'_{t,t-1+\ell}, Y_{i_1})\le  2 n^{-1/p} s  } \\
	&\qquad\qquad\qquad \ldots \1{d(Z'_{t,t-1+\ell}, Y_{i_v})\le  2 n^{-1/p} s  }		\big]
	\end{split}
\end{align}
for pairwise disjoint indices $i_\ell$. The last inequality can be derived as follows: The number of different observations $v$ is in $ [k,wk]$. Given $v$ pairwise different indices $i_1,\ldots,i_v$ (observations $Y_{i_1},\ldots,Y_{i_v}$) each occurs with multiplicity $u_i$ such that $\sum_{i=1}^v u_i = wk$.

Applying finally the same reasoning as in the proof of Lemma~\ref{L:SimplicesAndMeasure} shows that \eqref{EQ:ConvergenceExpectationContinuous18} is bounded above by a universal constant $C_{s,p,w}$, which depends on $s,p,w$ but not on $n,\ell,k$. Clearly, the $w$th moment of the simplex count in the second line in \eqref{EQ:ConvergenceExpectationContinuous17} is at most $C_{s,p,w}$, too.

We can now return to \eqref{EQ:ConvergenceExpectationContinuous16}. Set $C'_{w} = (\sum_{j=1}^\infty j^{-aw} )^{1/w}$ for some $a>1/w$. Then relying on the coupling result from \eqref{EQ:ConvergenceExpectationContinuous9}, \eqref{EQ:ConvergenceExpectationContinuous16} is at most
\begin{align*}
		&4 C_{s,p,w}^{1/w} \sum_{j=t+m}^{n} (j-t) \ \mathbb{E}\Big[ \1{\tau_c(t) = j } \1{ {Z'}_{t,t} \neq \wt{Z}_{t,t} } \Big]^{1/\ol w}  \\
		& \le 4 C_{s,p,w}^{1/w} \sum_{j=t+m}^{n} (j-t) \ \p( \tau_c(t) = j | {Z'}_{t,t} \neq \wt{Z}_{t,t} )^{1/\ol w} \ \p( {Z'}_{t,t} \neq \wt{Z}_{t,t} )^{1/\ol w}  \\
		&\le 4 C_{s,p,w}^{1/w} \ (\epsilon/4)^{1/\ol w} \ \Big( \sum_{j=t+m}^{n} (j-t)^{-a w} \Big)^{1/w} \ \Big( \sum_{j=t+m}^{n} (j-t)^{\ol w(a+1) } \ \p( \tau_c(t) = j | {Z'}_{t,t} \neq \wt{Z}_{t,t} ) \Big)^{1/\ol w} \\
		&\le 4 C_{s,p,w}^{1/w} \ C'_w \ (\epsilon/4)^{1/\ol w} \ \E{ (\tau_c(t) - t)^{\ol w (a+1) } \ | \ \{ {Z'}_{t,t} \neq \wt{Z}_{t,t} \}  }^{1/\ol w} \\
		&= 4 C_{s,p,w}^{1/w} \ C'_w \ (\epsilon/4)^{1/\ol w} \ \Big\{ \int_t^\infty (u-t)^{\ol w (a+1) - 1} \ \p( \tau_c(t)  > u \ | \ {Z'}_{t,t} \neq \wt{Z}_{t,t} ) \ \intd{u} \Big\}^{1/\ol w}.
\end{align*}
Relying on \eqref{EQ:ConvergenceExpectationContinuous2},  \eqref{EQ:ConvergenceExpectationContinuous14} and \eqref{EQ:ConvergenceExpectationContinuous15} this last term is of order $\epsilon^{1/\ol w}$. This shows that  \eqref{EQ:ConvergenceExpectationContinuous16} is of order $\epsilon^{1/w}$ uniformly in $t\in\{1,\ldots,n\}$ and $n$.

In order to complement the considerations following \eqref{EQ:ConvergenceExpectationContinuous16}, it remains to consider the restriction to the event $\{ \tau_c(t) > n \}$ for $t\in\{1,\ldots,n\}$. Here we need additionally to consider the average over all $t$
\begin{align}\label{EQ:ConvergenceExpectationContinuous19}
				&n^{-1} \sum_{t=1}^n \E{ \1{\tau_c(t) > n } \Big| \beta^{r,s}_q( \cK( n^{1/p} {Z'}_{t,1}^{t,n} ) ) - \beta^{r,s}_q( \cK( n^{1/p} \wt{Z}_{t,1}^{t,n} ) ) \Big| } 
\end{align}
Carrying out similar calculations as in \eqref{EQ:ConvergenceExpectationContinuous17} and \eqref{EQ:ConvergenceExpectationContinuous18} it is not difficult to see that \eqref{EQ:ConvergenceExpectationContinuous19} is at most
\begin{align*}
	&4 C_{s,p,w}^{1/w} \ n^{-1} \sum_{t=1}^n (n-t) \ \p( \tau_c(t) > n  \ | \ {Z'}_{t,t} \neq \wt{Z}_{t,t} )^{1/\ol w} \ \p( {Z'}_{t,t} \neq \wt{Z}_{t,t} )^{1/\ol w} \\
	&\le 4 C_{s,p,w}^{1/w} \ (\epsilon/4)^{1/\ol w} \ n^{-1} \sum_{t=1}^n (n-t) \ \p( \tau_c(t) - t > n - t \ | \  {Z'}_{t,t} \neq \wt{Z}_{t,t} )^{1/\ol w}\\
	&\le 4 C_{s,p,w}^{1/w} \ (\epsilon/4)^{1/\ol w} \ n^{-1} \sum_{u=0}^\infty u \ \p( \tau_c(t) - t > u \ | \  {Z'}_{t,t} \neq \wt{Z}_{t,t} )^{1/\ol w}.
\end{align*}
Using once more the result in \eqref{EQ:ConvergenceExpectationContinuous14},  \eqref{EQ:ConvergenceExpectationContinuous15}  and \eqref{EQ:ConvergenceExpectationContinuous2} yields directly that this last upper bound vanishes.

Combining these results yields \eqref{EQ:ConvergenceExpectationContinuous13}. This completes the proof.
\end{proof}

\begin{proof}[Proof of Corollary~\ref{C:VagueConvergence}]
It is shown in Proposition~3.4 in \cite{hiraoka2018limit} that the pointwise convergence of persistent Betti numbers implies the vague convergence of the corresponding sequence of persistent diagrams.
\end{proof}

\subsection{Technical details on Section~\ref{Sec_ExtensionsToRandomFields}}
\begin{proof}[Proof of Theorem~\ref{T:ConvergenceExpectationDiscreteRandomField}]
The statement follows immediately from Theorem~\ref{T:ConvergenceExpectationDiscrete}.
\end{proof}

\begin{proof}[Proof of Theorem~\ref{T:ConvergenceExpectationMarkovRandomField}]
The proof is very similar to that of Theorem~\ref{T:ConvergenceExpectationContinuous} and we only study the main differences in detail. First, we construct an $\epsilon$-approximation $\wt{X}$ of $X$. For this purpose we consider the joint distribution of $\{X_u: u\le (1,\ldots,1)\}$, which is completely determined by the joint density $g\colon [0,1]^{2^d p}\to (0,\infty)$. Let $\epsilon>0$ and choose a discrete approximation $g_\epsilon\colon [0,1]^{2^d p}\to (0,\infty)$ of $g$ such that conditional densities $f_{\epsilon,s}$ ($s\in\{0,1\}^d$) are derived from $g_\epsilon$ in the same spirit as in the proof of Theorem~\ref{T:ConvergenceExpectationContinuous}; we refer to Figure~\ref{fig:FactorizationScheme}. These densities are strictly positive and satisfy
\[
			| f_{\epsilon,s}( x | y ) - f_s (x|y) |\le \epsilon, \quad \forall x\in[0,1]^p, \quad \forall y\in [0,1]^{\|s\|_1 p}, \quad \forall s\in\{0,1\}^d\setminus \{ (0,\ldots,0) \}
\]
as well as $\|f_{\epsilon,(0,\ldots,0)}- f_{(0,\ldots,0)}\|_\infty \le \epsilon$.
(This requirement is the analog to \eqref{EQ:ConvergenceExpectationContinuous1}). Obviously, the discrete (conditional) densities $f_{\epsilon,s}$ determine the random field $\wt{X}$ completely. Also, due to the blocked structure of the densities $f_{\epsilon,s}$ and the condition from \eqref{E:DecayDependenceMRF}, the random field $\wt{X}$ satisfies the requirements of Theorem~\ref{T:ConvergenceExpectationDiscreteRandomField}. Reasoning as in \eqref{EQ:ConvergenceExpectationContinuous3} to \eqref{EQ:ConvergenceExpectationContinuous5}, it is sufficient to study the difference
\begin{align}\label{EQ:ConvExpMRF1}
			\pi(N)^{-1} \E{ \beta_q^{r,s}(\cK( \pi(N)^{1/p}  \mX_N ) ) } - \pi(N)^{-1} \E{ \beta_q^{r,s}(\cK( \pi(N)^{1/p}  \wt{\mX}_N ) ) }
\end{align}
for an arbitrary but fixed $N\in\N^d$.

We use the same expansion for this difference as in the case of Markov chains, see \eqref{EQ:ConvergenceExpectationContinuous6}, \eqref{EQ:ConvergenceExpectationContinuous7} and \eqref{EQ:ConvergenceExpectationContinuous8}. But this time we use the ordering $>_d$ for the expansion. We obtain for each $u\in\N^d$ a coupling $( (Z'_{u,v}, \wt{Z}_{u,v}): v\in\N^d )$ with the properties
\begin{align*}
			(i)\quad & Z'_{u,w} = \wt{Z}_{u,w}, \quad \forall w <_d u; \text{ these are distributed according to the } f_{\epsilon,s}, \; s\in\{0,1\}^d. \\
			(ii)\quad& \text{For the position $u\in\N^d$ let $s\in\{0,1\}^d$ be the associated index. Then}\\
			\quad & Z'_{u,u} \text{ is distributed according to $f_s$, } \wt{Z}_{u,u} \text{ is distributed according to} f_{\epsilon,s} \text{ and }\\
			\quad& \p(  Z'_{u,u} \neq \wt{Z}_{u,u} \,|\, Z'_{u,w} = \wt{Z}_{u,w}, \; \forall w <_d u) \le \epsilon/ 4 \quad a.s.\\
			(iii)\quad& \p( Z'_{u,v} \neq \wt{Z}_{u,v} | Z'_{u,w} = \wt{Z}_{u,w}, \; w <_d u, Z'_{u,u}, \wt{Z}_{u,u} ) \le \Gamma^{(\infty)}_{u,v} \quad a.s., \quad \forall v >_d u, \\
			&\qquad \text{and the } Z'_{u,v}, \wt{Z}_{u,v} \text{ are distributed according to the } f_{s}, \; s\in\{0,1\}^d \text{ for all } v >_d u.
\end{align*}
Consequently, using that $\wt{Z}_{u, v} = {Z'}_{u, v}$, for all $v <_d u$, we can write the difference in \eqref{EQ:ConvExpMRF1} as
\begin{align}\begin{split}\label{EQ:ConvExpMRF2}
			&\pi(N)^{-1} \sum_{\substack{u\in\N^d: u\le N}} \mathbb{E} \Big[ \beta^{r,s}_q( \cK( \pi(N)^{1/p} \{{Z'}_{u, v }: v\le N  \}) ) \\
			&\qquad\qquad \qquad\qquad \quad - \beta^{r,s}_q( \cK( \pi(N)^{1/p} \{ \wt{Z}_{u, v}: v\le N \}) )  \Big].
\end{split}\end{align}
Given a location $u$ and a coupling $(Z'_{u,\cdot}, \wt{Z}_{u,\cdot})$, we define the coupling time
\begin{align*}
		\tau_c(u) &\coloneqq \inf\big\{	k\ge 0 \,|\, Z'_{u,v} = \wt{Z}_{u,v}, \forall v\in\N^d \text{ such that } \|u-v\|_{\max} = k \text{ and } v\ge u \big\}.
\end{align*}
So $\tau_c(u)$ is determined by the causal dependence pattern which is derived from the factorization of the joint distribution according to the ordering $>_d$. Note that both random fields $Z'_{u,\cdot}$ and $\wt{Z}_{u,\cdot}$ move in lockstep after $\tau_c(u)$.

Consider the tail of the coupling time $\tau_c$ at location $u$
\begin{align}
				&\p( \tau_c(u) > k \,|\, Z'_{u,u} \neq \wt{Z}_{u,u} ) \nonumber \\
				&= \p(  Z'_{u,v} \neq \wt{Z}_{u,v} \text{ for one } v\ge u \text{ with } \|v-u\|_{\max} = k  \,|\,  Z'_{u,u} \neq \wt{Z}_{u,u} ) \nonumber\\
				&\le \sum_{\substack{v: \; v\ge u\\ \|v-u\|_{\max} = k}} \p(  Z'_{u,v} \neq \wt{Z}_{u,v} \text{ for one } v\ge u \text{ with } \|v-u\|_{\max} = k  \,|\, Z'_{u,u} \neq \wt{Z}_{u,u} ) \nonumber \\
				&\le  \sum_{\substack{v: \; v\ge u\\ \|v-u\|_{\max} = k}} \Gamma^{(\infty)}_{u,v} \le c_d k^{d-1} \sup_{\substack{v: \; v\ge u\\ \|v-u\|_{\max} = k}} \Gamma^{(\infty)}_{u,v} \label{EQ:ConvExpMRF3}
\end{align}
for a constant $c_d\in\R_+$, which depends on $d$ but not on $u,k,N$.

Choose $ w \in \N$ such that we have with the abbreviation $\ol w = w/(w-1)$ that $(\delta/3+1)/\ol w - d > 1/w>0$, which is possible because by assumption $\delta> 3(d-1)$.

In the following, we will consider such a single difference in \eqref{EQ:ConvExpMRF2} and show that it is of order $\epsilon^{1/\ol w}$ uniformly in $N$; the calculations follow in a similar spirit as in the proof of Theorem~\ref{T:ConvergenceExpectationContinuous}, see \eqref{EQ:ConvergenceExpectationContinuous16} to \eqref{EQ:ConvergenceExpectationContinuous18}, so we omit some details. Clearly, we can restrict our considerations to the event $\{Z'_{u,u} \neq \wt{Z}_{u,u} \}$. Again, use $C'_{w} = (\sum_{j=1}^\infty j^{-aw} )^{1/w}$ but this time for $a\in (1/w, (\delta/3+1)/\ol w - d)$. Then using a similar bound on simplex counts, there is a constant ${\ol C}_{s,p,d,w}$ such that
\begin{align}
		&\mathbb{E} \Big[ |\beta^{r,s}_q( \cK( \pi(N)^{1/p} \{{Z'}_{u, v }: v\le N  \}) )  - \beta^{r,s}_q( \cK( \pi(N)^{1/p} \{ \wt{Z}_{u, v}: v\le N \}) ) | \Big] \nonumber \\
		&\le \sum_{k=1}^{ \|N-u\|_{\max} } \mathbb{E} \Big[ \1{ \tau_c(u) = k, Z'_{u,u} \neq \wt{Z}_{u,u}  } \Big| \beta^{r,s}_q( \cK( \pi(N)^{1/p} \{{Z'}_{u, v }: v\le N  \}) ) \nonumber \\
				&\qquad\qquad\qquad - \beta^{r,s}_q( \cK(  \pi(N)^{1/p} \{ \wt{Z}_{u, v}: v\le N \}) ) \Big| \Big] \nonumber \\
				&\quad + \mathbb{E} \Big[ \1{ \tau_c(u) > \|N-u\|_{\max}, Z'_{u,u} \neq \wt{Z}_{u,u}  } \Big| \beta^{r,s}_q( \cK( \pi(N)^{1/p} \{{Z'}_{u, v }: v\le N  \}) ) \nonumber \\
				&\qquad\qquad\qquad - \beta^{r,s}_q( \cK(  \pi(N)^{1/p} \{ \wt{Z}_{u, v}: v\le N \}) ) \Big| \Big] \nonumber \\
				\begin{split} \label{EQ:ConvExpMRF4}
				&\le 4 {\ol C}_{s,p,d,w}^{1/w} \sum_{k=1}^{ \|N-u\|_{\max} }  k^d \ \p( \tau_c(u) = k \ | \ Z'_{u,u} \neq \wt{Z}_{u,u} )^{1/\ol w} \ \epsilon^{1/\ol w} \\
				&\quad + 4 {\ol C}_{s,p,d,w}^{1/w} \ \| N-u \|_{\max}^d \ \p( \tau_c(u) >  \|N-u\|_{\max} \ | \ Z'_{u,u} \neq \wt{Z}_{u,u} )^{1/\ol w} \epsilon^{1/\ol w}
				\end{split}
\end{align}
for all $u\le N$ and for all $N$.
Regarding the sum in \eqref{EQ:ConvExpMRF4}, we use the definition of $C'_w$ to see
\begin{align}
	&\sum_{k=0}^{ \|N-u\|_{\max} } k^d \ \p( \tau_c(u) = k \,|\, Z'_{u,u} \neq \wt{Z}_{u,u} )^{1/\ol w} \nonumber \\
	&\le \Big(\sum_{k=0}^{ \|N-u\|_{\max} } k^{-aw} \Big)^{1/w} \ \Big( \sum_{k=0}^{ \|N-u\|_{\max} } k^{(a+d)\ol w} \ \p( \tau_c(u) = k \,|\, Z'_{u,u} \neq \wt{Z}_{u,u} ) \Big)^{1/\ol w} \nonumber \\
	&\le C'_w \ \E{  \tau_c(u)^{ (a+d)\ol w} \,|\, \{ Z'_{u,u} \neq \wt{Z}_{u,u} \} }^{1/\ol w} \nonumber \\
	&=  C'_w \Big( \int_{0}^\infty s^{(a+d)\ol w - 1 } \p ( \tau_c(u) > s \,|\, Z'_{u,u} \neq \wt{Z}_{u,u} ) \intd{s} \Big)^{1/\ol w}. \label{EQ:ConvExpMRF5}
\end{align}
Applying the upper bound from \eqref{EQ:ConvExpMRF3} together with the condition from \eqref{E:DecayDependenceMRF} shows that \eqref{EQ:ConvExpMRF5} is uniformly bounded in $u$ and $N$. 

We use \eqref{EQ:ConvExpMRF4} together with \eqref{EQ:ConvExpMRF5} as well as \eqref{EQ:ConvExpMRF3} to give a bound on \eqref{EQ:ConvExpMRF2} (up to a universal multiplicative constant) as follows
\begin{align*}
		&\epsilon^{1/\ol w} \pi(N)^{-1} \sum_{u\in \N^d: u\le N} \big(1 +  \| N-u \|_{\max}^d \ \p( \tau_c(u) >  \|N-u\|_{\max} \ | \ Z'_{u,u} \neq \wt{Z}_{u,u} )^{1/\ol w}  \big) \\
		&\le \epsilon^{1/\ol w} + \epsilon^{1/\ol w} \ \pi(N)^{-1}  \sum_{k=0}^\infty \Big( \sum_{u\in \N^d: u\le N} \1{ \|N-u\|_{\max} = k } \p( \tau_c(u) > k \ | \ Z'_{u,u} \neq \wt{Z}_{u,u} )^{1/\ol w} \Big)\\
		&\le \epsilon^{1/\ol w} + \epsilon^{1/\ol w} \ \pi(N)^{-1}  \sum_{k=0}^\infty c_d k^{d-1} \ k^d \ \p( \tau_c(u) > k \ | \ Z'_{u,u} \neq \wt{Z}_{u,u} )^{1/\ol w} 
\end{align*}
because $(\tau_c(u) : u)$ is stationary. We can now repeat the calculations which lead to \eqref{EQ:ConvExpMRF5} to see that this last sum satisfies
\begin{align*}
	&\sum_{k=0}^\infty c_d k^{2d-1} \p( \tau_c(u) > k \ | \ Z'_{u,u} \neq \wt{Z}_{u,u} )^{1/\ol w} \\
	&\le C'_w \Big( \int_{0}^\infty s^{(a+[2d-1])\ol w - 1 } \p ( \tau_c(u) > s \,|\, Z'_{u,u} \neq \wt{Z}_{u,u} ) \intd{s} \Big)^{1/\ol w}. 
\end{align*}
Relying once more on \eqref{EQ:ConvExpMRF3} and \eqref{E:DecayDependenceMRF} shows then that this integral is uniformly bounded in $u$ and $N$.
This shows that \eqref{EQ:ConvExpMRF2} is of order $\epsilon^{1/\ol w}$. Consequently, \eqref{EQ:ConvExpMRF1} is of order $\epsilon^{1/\ol w}$, too.
\end{proof}

\appendix 
\section{McDiarmid inequalities for Marton couplings}\label{Appendix1}
In this section we study McDiarmid inequalities for Marton couplings. Notable contributions to this topic are \cite{samson2000concentration}, \cite{chazottes2007concentration}, \cite{kontorovich2008concentration}, \cite{redig2009concentration}. We shall first state a result of \cite{paulin2015concentration} who uses Marton couplings to characterize the dependence of the data.

\begin{definition}[Partition]\label{Def:Partition}
A partition of a random vector $Z=(Z_1,\ldots,Z_N)$ is a deterministic division of $Z$ into random variables $\hat{Z}_i$, $i=1,\ldots,n$ for some $n\le N$ such that the set $\{Z_1,\ldots,Z_N\}$ is partitioned by $(\hat{Z}_i)_{i=1,\ldots,n}$. Denote the number of elements of $\hat{Z}_i$ by $s(\hat{Z}_i)$ and write $s(\hat{Z})$ for the size of the partition which is $\max_{i=1,\ldots,n} s(\hat{Z}_i)$.
\end{definition}

\begin{theorem}[McDiarmid's inequality, \cite{paulin2015concentration}]\label{Thrm:McDiarmidIneq}
Let $Z=(Z_1,\ldots,Z_N)$ be a random variable in $\Lambda = \Lambda_1\times\cdots\times\Lambda_N$. Assume that $Z$ admits a partitioning $\hat{Z}=(\hat{Z}_1,\ldots,\hat{Z}_n)$ which allows a Marton coupling with mixing matrix $\Gamma\in\R^{n\times n}$. Let $\phi: \Lambda\rightarrow \R$ be Lipschitz continuous w.r.t.\ the Hamming distance, i.e., there is a $c= (c_1,\ldots,c_N)\in\R^N$ such that
\begin{align}\label{Eq:McDiarmidHammingCondition}
				|\phi(x)-\phi(y)|\le \sum_{j=1}^N c_j \1{x_j\neq y_j}, \quad x,y\in \Lambda.
\end{align}
Set $\cI_i \coloneqq \{j=1,\ldots,N: Z_j \in \hat{Z}_i \}$ and $C_i(c) \coloneqq \sum_{j\in \cI_i} c_j$ for $i=1,\ldots,n$. Then
\begin{align}\label{Eq:McDiarmidIneqLaplace}
		\log \E{	\exp\left(	\gamma \left(	\phi(Z)-\E{\phi(Z)}	\right)	\right) } \le \gamma^2 \norm{ \Gamma C(c) }^2 / 8.
\end{align}
In particular,
\begin{align}\label{Eq:McDiarmidIneq}
		\p\left( \left| \phi(Z)-\E{\phi(Z)} \right| \ge t \right) \le 2 \exp\left( - \frac{2t^2}{ \norm{ \Gamma C(c) }^2}		\right) \le 2 \exp\left( - \frac{2t^2}{ \norm{ \Gamma}^2 \norm{c }^2 s(\hat{Z}) }		\right).
\end{align}
\end{theorem}
The proof uses the following lemma of \cite{devroye2012combinatorial}:
\begin{lemma}
Let $\cF$ be a sub-$\sigma$-algebra, $U,V,W$ random variables which satisfy $U\le V\le W$ $a.s.$ Moreover, $U,W$ are $\cF$-measurable and $\E{V|\cF}=0$. Then
$$ 
	\log \E{\exp\left(\gamma V \right)|\cF} \le \gamma^2 (U-W)^2/8, \qquad  \gamma\in\R.
$$
\end{lemma}
\begin{proof}[Proof of Theorem~\ref{Thrm:McDiarmidIneq}]
We consider the natural filtration of the random vector $\hat{Z}$, i.e., $\cF_i = \sigma(\hat{Z}_1,\ldots,\hat{Z}_i)$ for $i=0,\ldots,n$ and define $\hat\phi ( (\hat{x_i})_i ) \coloneqq \phi( x)$ for $x\in \Lambda$. Then $\hat\phi$ is also Lipschitz continuous w.r.t.\ Hamming distance, more precisely,
$$
		|\hat\phi (x) - \hat\phi(y) | \le \sum_{j=1}^N c_j \1{x_j\neq y_j} \le \sum_{i=1}^n C_i(c) \1{\hat{x}_i \neq \hat{y}_i}. 
$$
Set $V_i \coloneqq \E{\hat\phi(\hat{Z})|\cF_i } - \E{\hat\phi(\hat{Z})|\cF_{i-1} }$ for $i=1,\ldots,n$. Moreover, define for $a\in \hat\Lambda_i = \prod_{j\in \cI_i} \Lambda_j$
\begin{align*}
		I_i(a) & \coloneqq \int_{\hat\Lambda_{i+1}\times\ldots \times\hat\Lambda_n} \p\left( \hat{Z}_{i+1} \in \intd{\hat{z}}_{i+1},\ldots, \hat{Z}_{n} \in \intd{\hat{z}}_{n} \,|\, \hat{Z}_{1},\ldots,\hat{Z}_{i-1}, \hat{Z}_{i}= a 	\right) \\
		&\qquad\qquad\qquad\quad \hat\phi\left(\hat{Z}_1,\ldots,\hat{Z}_{i-1},a,\hat{z}_{i+1},\ldots,\hat{z}_{n} \right) .
\end{align*}
And write $\nu_i$ for the conditional distribution of $\hat{Z}_i$ given $(\hat{Z}_1,\ldots,\hat{Z}_{i-1})$, i.e.,
$$
		\nu_i = \mathbbm{M}_{\hat{Z}_i|(\hat{Z}_1,\ldots,\hat{Z}_{i-1})} \left( (\hat{Z}_1,\ldots,\hat{Z}_{i-1}), \cdot \right).
$$
Then, it follows with elementary calculations that
$$
		V_i \le \esssup_{ \text{ w.r.t.\ } \nu_i } I_i(\cdot) - \essinf_{ \text{ w.r.t.\ } \nu_i} I_i(\cdot).
$$
Now let $\epsilon>0$ be arbitrary but fixed. Choose $a^*,b^*\in \hat\Lambda_i$ such that $I_i(a^*) \ge \esssup_{ \text{ w.r.t.\ } \nu_i } I_i(\cdot) - \frac{\epsilon}{2}$ and $I_i(b^*) \le \essinf_{\text{ w.r.t.\ } \nu_i  } I_i(\cdot) + \frac{\epsilon}{2}$. Next, we use the Marton coupling of $\hat{Z}$ to obtain 
\begin{align*}
		I_i(a^*) - I_i(b^*) &= \E{ \hat\phi\left(\hat{Z}^{(\hat{Z}_1,\ldots,\hat{Z}_{i-1},a^*,b^* )} \right) - \hat\phi\left(\hat{Z'}^{(\hat{Z}_1,\ldots,\hat{Z}_{i-1},a^*,b^* )} \right) \,\Big|\, \hat{Z}_1,\ldots,\hat{Z}_{i-1} } \\
		&\le \sum_{k=i}^n C_k(c) \, \p\left( \hat{Z}_k^{(\hat{Z}_1,\ldots,\hat{Z}_{i-1},a^*,b^* )} \neq \hat{Z'}_k^{(\hat{Z}_1,\ldots,\hat{Z}_{i-1},a^*,b^* )} \,\Big|\, \hat{Z}_1,\ldots,\hat{Z}_{i-1}	\right) \\
		&\le \sum_{k=i}^n C_k(c) \Gamma_{i,k}.
\end{align*}	
And as $\epsilon >0$ was arbitrary, 
$$
		\esssup_{ \text{ w.r.t.\ } \nu_i } I_i(\cdot) - \essinf_{ \text{ w.r.t.\ } \nu_i } I_i(\cdot) \le \sum_{k=i}^n C_k(c) \Gamma_{i,k}.
$$
Moreover, we have
$$
		\essinf_{ \text{ w.r.t.\ } \nu_i } I_i(\cdot) - \E{\hat\phi(\hat{Z})|\cF_{i-1} } \le V_i \le \esssup_{ \text{ w.r.t.\ } \nu_i  } I_i(\cdot)  - \E{\hat\phi(\hat{Z})|\cF_{i-1} }, \quad a.s.
$$
and both the left- and the right-hand-side are $\cF_{i-1}$-measurable. Consequently, using the lemma of \cite{devroye2012combinatorial}, we find that
$$
		\E{\exp\left(	\gamma V_i	\right)\Big| \cF_{i-1} } \le \exp\left(	\frac{\gamma^2}{8} \left(\sum_{k=i}^n C_k(c) \Gamma_{i,k} \right)^2	\right).
$$
This establishes the claim in \eqref{Eq:McDiarmidIneqLaplace}. The final result in \eqref{Eq:McDiarmidIneq} follows from the inequalities
$
			\norm{\Gamma C(c) }^2 \le \norm{\Gamma}^2 \norm{C(c)}^2 \le \norm{\Gamma}^2 \norm{c}^2 s(\hat{Z}).
$

\end{proof}

The next proposition is due to \cite{fiebig1993mixing} and a consequence of Goldstein's maximal coupling, \cite{goldstein1979maximal}. See also \cite{paulin2015concentration} Proposition 2.6 and \cite{samson2000concentration} Proposition 2.
\begin{proposition}[\cite{fiebig1993mixing}, p. 482, (2.1)]\label{Prop:Goldstein}
Let $P$ and $Q$ be two probability distributions on some common Polish space $\Lambda_1\times\ldots\times\Lambda_N$ both admitting a strictly positive density w.r.t.\ to a measure $\rho$. Then there is a coupling of random vectors $X=(X_1,\ldots,X_N)$, $Y=(Y_1,\ldots,Y_N)$ such that $\cL(X)=P$, $\cL(Y)=Q$ and 
$$
		\p(X_i\neq Y_i) \le d_{\text{TV}}( \cL(X_i,\ldots,X_N), \cL(Y_i,\ldots,Y_N)), \quad i=1,\ldots,N.
$$
\end{proposition}

\section*{Acknowledgments}
The author thanks an anonymous referee whose careful reading and detailed reports improved the manuscript considerably.
This research was supported by the German Research Foundation (DFG), Grant Number KR-4977/1-1.

\end{document}